\documentclass[11pt,a4paper]{article}

\usepackage{amssymb}
\usepackage{amsmath, amsthm}
\usepackage{arydshln}
\usepackage{epsfig}
\usepackage{setspace}
\usepackage{comment} 
\usepackage[margin=1in]{geometry}
\def\RR{{\mathbb{R}}}

\def\CC{{\mathbb{C}}}

\numberwithin{equation}{section}

\newcommand{\dom}{\mathop{\rm dom} }
\newcommand{\lv}{\mathop{\rm Lv} }

\newcommand{\val}{\mathop{\rm val}}

\newcommand{\ncrk}{\mathop{\rm ncrk }}

\newcommand{\diag}{\mathop{\rm diag} }

\DeclareMathOperator{\trace}{Tr}
\DeclareMathOperator{\rk}{rk}
\DeclareMathOperator{\spec}{spec}

\newcommand{\ad}{\mathop{\rm ad} }

\newtheorem{Thm}{Theorem}[section]
\newtheorem{Prop}[Thm]{Proposition}
\newtheorem{Lem}[Thm]{Lemma}

\newtheorem{Question}[Thm]{Question}
\newtheorem{Conj}[Thm]{Conjecture}
\theoremstyle{definition}
\newtheorem{Clm}{Claim}
\newtheorem{Def}[Thm]{Definition}

\newtheorem{Ex}[Thm]{Example}

\title{Generalized gradient flows in Hadamard manifolds \\and\\
convex optimization on entanglement polytopes}
\author{Hiroshi Hirai\footnote{
Graduate School of Mathematics,
Nagoya University, Nagoya, 464-8602, Japan.
\texttt{\footnotesize hirai.hiroshi@math.nagoya-u.ac.jp}
}}

\begin{document}

\maketitle
\begin{abstract}
In this paper, we address the optimization problem of minimizing $Q(df_x)$ over a Hadamard manifold ${\cal M}$, where $f$ is a convex function on ${\cal M}$, $df_x$ is the differential of $f$ at $x \in {\cal M}$, and $Q$ is a function on the cotangent bundle of ${\cal M}$. This problem generalizes the problem of minimizing the gradient norm $\|\nabla f(x)\|$ over ${\cal M}$, studied by Hirai and Sakabe FOCS2024. We formulate a natural class of $Q$ in terms of convexity and invariance under parallel transports, and introduce a generalization of the gradient flow of $f$ that is expected to minimize $Q(df_x)$. For basic classes of manifolds, including the product of the manifolds of positive definite matrices, we prove that this gradient flow attains $\inf_{x\in {\cal M}} Q(df_x)$ in the limit, and yields a duality relation. This result is applied to the Kempf-Ness optimization for GL-actions on tensors, which is Euclidean convex optimization on the class of moment polytopes, known as the entanglement polytopes. This type of convex optimization arises from tensor-related subjects in theoretical computer science, such as quantum functional, $G$-stable rank, and noncommutative rank.
\end{abstract}

MSC-classifications: 90C25, 53C35
\section{Introduction}
The present paper is a continuation of the work~\cite{Hirai2025,HiraiSakabe2024FOCS} studying ``divergent" gradient flows of lower-unbounded convex functions $f$ on Hadamard manifolds ${\cal M}$
and their applications to optimization problems hidden 
in the boundary ${\cal M}^{\infty}$ at infinity of ${\cal M}$.
In~\cite{HiraiSakabe2024FOCS},  Hirai and Sakabe established
a duality relation of the infimum of the gradient-norm $\|\nabla f(x)\|$ over $x \in {\cal M}$ by the use of the gradient flow  of $f$.
Specifically, if the infimum is positive, 
then it is equal to 
the supremum of the negative of the {\em recession function} 
$f^{\infty}$~\cite{Hirai_Hadamard2022,KLM2009JDG} of $f$, which is a function defined on the boundary ${\cal M}^{\infty}$. 
Further, the infimum and supremum are attained by 
the limit of the solution $x(t)$ of the gradient-flow ODE $\dot x(t) = -\nabla f (x(t))$, $x(0) = x_0$ as follows:
\begin{equation}\label{eqn:min_gradient-norm}
\lim_{t \to \infty} \|\nabla f(x(t))\| = \inf_{x \in {\cal M}} \|\nabla f(x)\| = 
\sup_{\xi \in {\cal M}^{\infty}} - f^{\infty}(\xi) = - f^{\infty} \left(\lim_{t\to \infty}x(t)\right), 
\end{equation}
where the convergence $\lim_{t \to \infty} x(t) \in {\cal M}^{\infty}$ is with respect to the cone topology of 
the visual compactification ${\cal M} \cup {\cal M}^{\infty}$.

Such an investigation originates from 
{\em operator scaling}~\cite{Gurvits2004} and its generalizations, which
have brought about remarkable developments and numerous applications in recent years; see e.g., \cite{BFGOWW_FOCS2018,BFGOWW_FOCS2019,GGOW} and the references therein. They involve optimization of the {\em Kempf-Ness functions} 
associated with reductive group actions, which is formulated as 
(possibly lower-unbounded) convex optimization on Hadamard manifolds.
Hirai and Sakabe~\cite{HiraiSakabe2024FOCS} utilized (\ref{eqn:min_gradient-norm}) to 
reveal the asymptotic behavior of the gradient flow and its discretization ({\em gradient descent}) for operator scaling. 
They also revealed the link with the theory of {\em moment-weight inequality} 
(Georgoulas, Robbin, and Salamon~\cite{GRS_Book}), which provides 
differential geometric understanding of stability in geometric invariant theory (GIT). 
In a follow-up paper, Hirai~\cite{Hirai2025} showed a refinement of (\ref{eqn:min_gradient-norm}) 
for the Hadamard manifold ${\cal P}_n$ of positive definite matrices ({\em PD-manifold}), 
in which the norm $\| \cdot \|$ is a general GL-invariant Finsler norm and $x(t)$ is  a non-differentiable generalization~(a {\em curve of maximal slope}~\cite{Ambrosio_Book}) of gradient flow.
This leads to a new characterization of the {\em noncommutative rank}~\cite{IQS2017} 
in terms of the trace norm of the scaling limit of operator scaling.
%
%a class of related nonpositively curved metric space
%Gromov hyperbolic space

In the light of these developments, we actively interpret (\ref{eqn:min_gradient-norm}) 
as an optimization framework
of minimizing the gradient-norm $\|\nabla f(x)\|$ over ${\cal M}$, and consider the following generalization.
{\em Given a function $Q$ on the cotangent bundle of ${\cal M}$, 
minimize the $Q$-value of 
the differential $df_x$ of $f$ over $x \in {\cal M}$:}
\begin{equation}\label{eqn:minimum_Q-gradient-norm_problem}
\mbox{Min.} \quad Q(df_x) \quad {\rm s.t.} \quad x \in {\cal M}.
\end{equation}
We call this problem the {\em minimum $Q$-gradient-norm problem}.
It is clear that 
the above gradient-norm setting is obtained by taking the dual norm $Q := \|\cdot\|^*$. % since $\|\nabla f(x)\| = \|df_x\|^*$.
We then address the following questions:
\begin{itemize}
  \item[1.] Which does $({\cal M},Q)$ admit a duality of form 
  $\displaystyle \inf_{x \in {\cal M}} Q(df_x) = \sup \ [\, \cdots]$?
  \item[2.] Is there an analogue of gradient flow $x(t)$ that minimizes $Q(df_x)$ in its limit?
  \item[3.] Can we explore further applications? 
\end{itemize}

The main contribution of this paper is a partial but 
satisfactory answer to these questions.
We formulate a natural class of functions $Q$ in terms of convexity and 
invariance under parallel transports (or {\em holonomy group}), 
and introduce a generalization of the gradient flow, called the {\em $Q$-gradient flow}, 
that is expected to minimize $Q(df_x)$ in the limit. 
We establish in Theorem~\ref{thm:main}
an analogue of (\ref{eqn:min_gradient-norm}) for basic classes of Hadamard manifolds 
(Euclidean space, generic Hadamard manifold, PD-manifold, and their products), 
although we could not prove it for the general case (Conjecture~\ref{conj}).

This incomplete result has promising applications. 
Indeed, it is
applicable to the Kempf-Ness optimization for GL-actions on tensors, that widely appears in quantum information and theoretical computer science, 
such as quantum marginal problem, tensor rank, and matrix multiplication; see~\cite{Burgisser_ACT_Book,BFGOWW_FOCS2018,Landsberg_tensors,Landsberg_tensors2019}.
In this setting, 
the minimum $Q$-gradient-norm problem~(\ref{eqn:min_gradient-norm})  
becomes Euclidean convex optimization on the {\em moment polytope} $\Delta(v)$~\cite{GuilleminSternberg1982-I,GuilleminSternberg1984-II,Kirwan1984} 
for the GL-action on a tensor $v$. This class of moment polytopes is particularly called the {\em entanglement polytopes} (Walter, Doran, Gross, Christandl~\cite{WDGC2013_Science}) for their importance of studying entanglements in quantum multi-particle systems.
Convex optimization on the entanglement polytope $\Delta(v)$ has gained attention for 
its connections with tensor-rank subjects.
However, there is no good inequality description for moment polytopes, 
and optimizing on $\Delta(v)$ appears to be quite difficult; see the recent work~\cite{vdBCLNWZ_STOC2025} for this direction. 
Our result implies an inf-sup duality relation (Theorem~\ref{thm:duality_entanglement})
for optimization problems on $\Delta(v)$ with symmetric convex objectives. 
It is expected to lead to a good characterization 
(NP $\cap$ co-NP characterization) for this problem. 
%Also the dual problem is written as convex optimization 
%on the {\em Euclidean building} arising as the boundary of the product of PD-manifolds. 
Further, the convergence of the $Q$-gradient flow 
gives rise to a sequence of points in $\Delta(v)$
accumulating to a minimizer (Theorem~\ref{thm:moment-limit}).
This generalizes the {\em moment-limit theorem}~\cite{GRS_Book}, 
and suggests a simple (sub)gradient-type algorithm solving this problem.
We demonstrate how our theory is applied to concrete examples ({\em quantum functional}~\cite{CristandlVramaZuiddam2023}, 
{\em $G$-stable rank}~\cite{Derksen2022ANT}, and noncommutative rank).
It is an important challenge to 
discretize the $Q$-gradient flow and develop a unified algorithm for these problems.

The rest of the paper is organized as follows.
In Section~\ref{sec:preliminaries}, we provide the necessary backgrounds on Riemannian geometry, Hadamard manifold, convex analysis, and PD-manifold.   
In Section~\ref{sec:Q-gradient_flow}, 
we introduce the $Q$-gradient flow, and establish the duality of 
the minimum $Q$-gradient-norm problem for the above-mentioned classes of $({\cal M},Q)$.
In Section~\ref{sec:entanglement_polytope}, 
we present applications to convex optimization on entanglement polytopes.

\section{Preliminaries}\label{sec:preliminaries}

We will use the standard notions of Riemannian geometry; see e.g.,\cite{Sakai1996}. See also \cite{Boumal_Book} for optimization perspective. 
We consider a complete Riemannian manifold ${\cal M}$ 
equipped with Riemannian metric $\langle, \rangle$ and Levi-Civita connection $\nabla$, where
$T_x = T_x{\cal M}$ and 
$T_x^* = T_x^*{\cal M}$ denote the tangent and cotangent spaces at $x \in {\cal M}$, respectively.
Let $T{\cal M} = \bigsqcup_{x \in {\cal M}}T_x$ and $T^*{\cal M} = \bigsqcup_{x \in M} T_x^*$
denote the tangent and cotangent bundles of ${\cal M}$, respectively.
The norm $\|\cdot\|$ on $T_x$ is defined by 
$\|v\| = \sqrt{\langle v,v \rangle}$ $(v \in T_x)$, and 
the dual norm $\|\cdot\|^*$ on $T_x^*$ is defined by
$\|p\|^* := \sup_{v \in T_x:\|v\|=1} p(v)$ $(p \in T_x^*)$.
For a curve 
$\gamma: [0, \ell] \to {\cal M}$, let $\dot \gamma(t) \in T_{\gamma(t)}$ denote the tangent vector of $\gamma$ at $t$. 
For a function $f$ on ${\cal M}$,
the differential $df_x \in T_x^*$ of $f$ at $x \in {\cal M}$ is given 
by $df_x(v) := (d/dt)\mid_{t = 0} f(\gamma(t))$, where $\gamma$ is any curve with $\gamma(0) = x$ and $\dot \gamma(0) = v$.
The gradient $\nabla f(x) \in T_x$ is given by $\langle \nabla f(x),v\rangle = df_x(v)$ $(v \in T_x)$.
For $t \in [0,\ell]$,
let $\tau_{\gamma}^t:T_{\gamma(0)} \to T_{\gamma(t)}$ denote the parallel transport
along $\gamma$ with respect to $\nabla$. 
The inverse map $\tau_{\gamma}^{-t}:= \tau_{\gamma^{-1}}^{\ell - t}$
is obtained by the reverse curve $t \mapsto \gamma^{-1}(t) := \gamma(\ell -t)$.
The parallel transport $\tau_{\gamma}^t:T^*_{\gamma(0)} \to T^*_{\gamma(t)}$ on the cotangent spaces
is defined by $\tau_{\gamma}^{t}(p)(\cdot) := p(\tau^{-t}_{\gamma} (\cdot))$
for $p \in T^{*}_{\gamma(0)}$.
For a vector/covector field $V = (V_x)_{x \in {\cal M}}$, 
the covariant derivative $\nabla_{u} V$ by $u \in T_x$ 
is given 
by $\nabla_{u}V := 
(d/dt)\mid_{t = 0} \tau_{\gamma}^{-t} V_{\gamma(t)}$, where $\gamma$ is any curve 
with $\gamma(0) = x$ and $\dot\gamma(0) = u$.
The length of a curve $\gamma:[0,\ell] \to {\cal M}$ is defined as $\int_{0}^{\ell} \|\dot \gamma(t)\|dt$.
The distance $d(x,y)$ of $x,y \in {\cal M}$ is defined as the infimum of the length of a curve connecting $x,y$.
By completeness, the infimum is attained
by a {\em geodesic}, which is a curve $\gamma$ with $\nabla_{\dot \gamma(t)}\dot\gamma = 0$.

\subsection{Hadamard manifold}

A {\em Hadamard manifold} is a simply-connected and complete Riemannian manifold 
having nonpositive sectional curvature.
Examples of Hadamard manifolds include 
Euclidean space, hyperbolic space, 
and the manifold of positive definite matrices (Section~\ref{subsec:psd}).
In a Hadamard manifold ${\cal M}$, 
a geodesic connecting two points is uniquely determined 
(up to affine rescaling).
For $u \in T_x$, there is a unique geodesic ray
$\gamma:[0,\infty) \to {\cal M}$ 
such that $\gamma(0) = x$ and $\dot\gamma(0) = u$. 
We denote $\gamma(t)$ particularly by $\exp_x tu$.
If $\|u\| =1$, then it is called a unit-speed geodesic ray. 
The exponential map $u \mapsto \exp_x u$ is a diffeomorphism from $T_x$ to ${\cal M}$.
The parallel transport along the geodesic from $x$ to $y$ 
is simply denoted by $\tau_{x\to y}$.

A function $f: {\cal M} \to \RR$ is said to be {\em geodesically convex} 
if $f \circ \gamma$ is convex for any geodesic $\gamma:[0,1] \to {\cal M}$:
\[
f(\gamma(t)) \leq (1-t) f(\gamma(0)) + t f(\gamma(1)) \quad (t \in [0,1]).
\]
If $f$ is twice differentiable, 
then the geodesic convexity is equivalent to the property that $(d^2/dt^2)(f\circ \gamma)(t) \geq 0$ for any geodesic $\gamma$. 
This condition is rephrased that 
the {\em covariant Hessian} $(u,v) \mapsto (\nabla df)(u,v) := (\nabla_{u} df)(v)$  
is positive semidefinite for every point; see \cite[Theorem 11.23]{Boumal_Book}.

Two unit-speed geodesic rays $\gamma,\gamma'$ are said to be {\em asymptotic} 
if $\sup_{t \in \RR_{\geq 0}} d(\gamma(t),\gamma'(t)) < + \infty$.
The asymptotic relation is an equivalence relation.
The set ${\cal M}^{\infty}$ of equivalence classes of all unit-speed rays is called the {\em boundary of {\cal M} at infinity}; see \cite[Chapter II.8]{BridsonHaefliger1999}.
We consider its {\em Euclidean cone} $C{\cal M}^{\infty} := ({\cal M}^{\infty} \times \RR_{\geq 0})/\sim$, 
where $(\xi,r) \sim (\xi',r')$ if and only if $r=r'=0$ 
or $(\xi,r) =(\xi',r')$.
Fix arbitrary $x_0 \in {\cal M}$.
Then, it is known that $u \mapsto (t \mapsto \exp_{x_0}t u) = ((t \mapsto \exp_{x_0}t u/\|u\|), \|u\|) $ 
is a bijection from $T_{x_0}$ to $C{\cal M}^{\infty}$. 
This induces a topology on $C{\cal M}^{\infty} \simeq T_{x_0}$, called the {\em cone topology}, which is independent of $x_0$. 
The boundary cone $C{\cal M}^{\infty}$ also has a geodesic metric determined 
from the {\em Tits metric} on ${\cal M}^{\infty}$; see~\cite[Chapter II.9]{BridsonHaefliger1999}.
This metric gives a different topology from the cone topology but makes $C{\cal M}^{\infty}$ a {\em Hadamard space}---a metric-space generalization 
of Hadamard manifolds and enjoys the unique geodesic property.  
So we can consider geodesic convexity on the space $C{\cal M}^{\infty}$; 
see e.g., \cite{Hirai_Hadamard2022}.

Let $f:{\cal M} \to \RR$ be a geodesically convex function. 
To study the behavior of $f$ at the boundary $C{\cal M}^{\infty}$ of ${\cal M}$,
define
the {\em recession function} $f^{\infty} = f_{x_0}^{\infty}:C{\cal M}^{\infty} \to \RR \cup \{\infty\}$ of $f$~\cite{Hirai_Hadamard2022,KLM2009JDG} by
\begin{equation}
f^{\infty}(\xi) := \lim_{t \to \infty}\frac{f(\exp_{x_0} t\xi)- f(x_0)}{t} \quad (\xi \in C{\cal M}^{\infty} \simeq T_{x_0}), 
\end{equation}
where $f^{\infty}$ is independent of $x_0$. That is, 
if $t \mapsto \exp_x t u$ and $t \mapsto \exp_y t v$ represent the same point in $C{\cal M}^{\infty}$, 
then $f^{\infty}_{x}(u) = f^{\infty}_{y}(v)$. 
%By this reason, we may omit subscript $x_0$.
It is known~\cite{Hirai_Hadamard2022} that $f^{\infty}$ is geodesically convex on the  Hadamard space $C{\cal M}^{\infty}$.

To define norm-like functions $Q$ that behave well under parallel transports,
we will utilize the notion of 
the {\em (restricted) Holonomy group} 
\[
Hol_{x_0}({\cal M}) := \{\tau_\gamma^{1} \in SO(T_{x_0}) 
\mid \mbox{$\gamma:[0,1] \to {\cal M}:\gamma(0)=\gamma(1) =x_0$: piecewise loop at $x_0$}\},
\]
where $SO(T_{x_0})$ denotes the rotation group keeping the norm $\|\cdot\|$ invariant.
See e.g., \cite[Chapter III.6]{Sakai1996} for holonomy groups. 
The holonomy group $Hol_{x_0}({\cal M})$ also acts on $T_{x_0}^*$ 
by $(\tau_{\gamma}^1,p) \mapsto \tau_{\gamma}^1 p (= p(\tau_\gamma^{-1}(\cdot)))$.
%A Hadamard manifold ${\cal M}$ is said to be {\em generic} if 
%$Hol_{x_0}({\cal M}) = SO(T_{x_0})$. 
%For example, hyperbolic space is generic.
%On the other hand, the holonomy group of Euclidean space is the trivial group.
The holonomy group is the product of those for de Rham factors of ${\cal M}$.
The {\em Berge holonomy theorem}~\cite[Theorem 10.3.4]{Petersen} implies 
that if ${\cal M}$ is irreducible, then 
the holonomy group acts transitively on the sphere of $T_x$ 
or ${\cal M}$ is a symmetric space with rank at least $2$.

\subsection{Convex analysis (on Hermitian matrices)}\label{subsec:convex-analysis}
We will utilize convex analysis on a finite-dimensional real vector space ${\cal V} (\simeq \RR^n)$; 
see e.g.,\cite{Rockafellar}. 
Let $\overline{\RR} := \RR \cup \{\infty\}$. 
A function $F: {\cal V} \to \overline{\RR}$ is said to 
be convex if $\lambda F(u) + (1- \lambda)F(v) \geq F(\lambda u + (1- \lambda) v)$
for $u,v \in {\cal V}$ and $\lambda \in [0,1]$.
Let $\dom F := \{u \in {\cal V}: F(u) < \infty\}$ denote the domain of $F$.
For $r \in \RR$, let $\lv(F,r) := \{u \in {\cal V}: F(u) \leq r\}$ denote the level set of $F$.
For a subset $C \subseteq {\cal V}$, let $\delta_C:{\cal V} \to \{0,\infty\}$ 
denote the indicator function defined by $\delta_C(u) := 0$ if $u \in C$ and $\infty$ otherwise.  
For a norm $\|\cdot\|$  and $v \in {\cal V}$, let 
$B_{\| \cdot \|}(v,r) := \{u \in {\cal V} \mid \|u-v\|\leq r\}$ 
denote the closed ball with respect to $\|\cdot\|$.

We (formally) consider the Legendre-Fenchel conjugate 
$F^*:{\cal V}^* \to \overline{\RR}$ in the dual space ${\cal V}^*$ of ${\cal V}$
as
\[
F^*(p) := \sup_{u \in {\cal V}} p(u) - F(u) \quad (u \in {\cal V}^*). 
\]
If $F$ is lower-semicontinuous convex, then ${F}^{**} = F$. 
The subdifferential $\partial F(u)$ of $F$ at $u \in {\cal V}$ is given by
\begin{eqnarray}
\partial F(u) &:= & \{ p \in {\cal V}^*  \mid F(v) \geq F(u) + p(v-u)\ (\forall p \in {\cal V})\} \nonumber \\
&=& \{ p \in {\cal V}^*  \mid F(u) + F^*(p) = p(u)\}. \label{eqn:subdifferential}
\end{eqnarray}
An element $p$ in $\partial F(u)$ is called a subgradient of $F$ at $u$.
If $F$ is differentiable at $u$, then $\{dF_u\} = \partial F(u)$.
%Also, $p \in \partial F(u)$ if and only if $p(u) = F(u) + F^*(p)$. 

In this paper, convex analysis on the space of Hermitian matrices (Lewis~\cite{Lewis1996}) 
will play important roles.
For a square complex matrix $A$, let $A^{\dagger}$ denote the conjugate transpose of $A$. 
For $\lambda \in \RR^n$, let $\diag \lambda$ denote the diagonal matrix 
having diagonals $\lambda_1,\lambda_2,\ldots,\lambda_n$ in order.  
Let ${\cal H}_{n}$ denote the real vector space of $n \times n$ Hermitian matrices.
The dual space ${\cal H}_n^*$ is identified with ${\cal H}_n$ by $X(Y) := \trace XY$. 
A convex function $F:{\cal H}_n \to \overline{\RR}$ 
is said to be {\em unitarily invariant} 
if $F(k X k^{\dagger}) = F(X)$ 
for any $X \in {\cal H}_n$ and unitary matrix $k \in U_n$.
In this definition, the unitary group
$U_n$ may be replaced by the special unitary group $SU_n$.  
A convex function $f:\RR^n \to \overline{\RR}$ is said to be {\em symmetric} if $f(x) = f(\sigma x)$ 
for any $x \in \RR^n$ and permutation matrix $\sigma$.
For a symmetric convex function $f$, 
the function $F_f:{\cal H}_n \to \overline{\RR}$ defined by 
$F_f(k (\diag \lambda) k^{\dagger}) := kf(\lambda)k^\dagger$ $(k \in SU_n,\lambda \in \RR^n)$ is unitarily invariant convex.
Conversely, 
any unitarily invariant convex function is written in this way~\cite{Davis1957}.
The conjugate $F_f^*$ is also unitarily invariant, and is equal to $F_{f^*}$~\cite{Lewis1996}.  
A subgradient of $F$ at $X$ can be computed by that of $f$ at $\lambda$; see \cite[Section 3]{Lewis1996} for details.
Specifically, if $F$ is differentiable, 
then $dF_X = k^{\dagger} (\diag df_{\lambda}) k$ for $\lambda \in \RR^n$, $k \in U_n$ with
$X= k (\diag \lambda) k^\dagger$.
These arguments are naturally extended to convex functions on
the product of ${\cal H}_{n_1},{\cal H}_{n_2},\ldots {\cal H}_{n_k}$ that are invariant 
under the action $X_i \mapsto k X_i k^{\dagger}$ of the unitary group $U_{n_i}$
for each component ${\cal H}_{n_i}$.
\begin{Ex}[Unitarily invariant norms $\|\cdot\|_{\rm op}$ and $\|\cdot \|_{\rm tr}$]\label{ex:norm}
The operator norm $\| \cdot \|_{\rm op}$ (the maximum of singular values) on ${\cal H}_n$
is unitarily invariant convex, and written as $F_{\|\cdot\|_{\infty}}$ 
for the $\ell_\infty$-norm $\|\cdot \|_{\infty}$ on $\RR^n$.
The trace norm $\|\cdot\|_{\rm tr}$ (the sum of singular values)
is unitarily invariant convex, and written as 
$F_{\|\cdot\|_{1}}$ for the $\ell_1$-norm $\|\cdot \|_{1}$.
Their conjugates are given by 
$(\| \cdot \|_{\rm op})^* = \delta_{B_{\|\cdot\|_{\rm tr}}(0,1)}$ and
$(\| \cdot \|_{\rm tr})^* = \delta_{B_{\|\cdot\|_{\rm op}}(0,1)}$.
The Frobenius norm $\|\cdot\|_{\rm F}$ is unitarily invariant convex, 
is written as $F_{\|\cdot\|_2}$ for the $\ell_2$-norm $\|\cdot \|_{2}$, where $(\| \cdot \|_{\rm F})^* = \delta_{B_{\|\cdot\|_{\rm F}}(0,1)}$.
\end{Ex}
Let ${\cal H}_n^+ \subseteq {\cal H}_n$ denote the set of $n \times n$ positive semidefinite matrices, and
let ${\cal D}_n^+ \subseteq {\cal H}_n^+$ denote the set of positive semidefinite matrices 
$X$ with $\trace X = 1$, i.e., {\em density matrices}.
%\begin{equation}
%{\cal D}^+_n = {\cal D}_n \cap {\cal H}_n^+.
%\end{equation}
%
\begin{Ex}[von Neumann entropy]\label{ex:entropy}
The {\em von Neumann entropy} $H: {\cal H}_n \to \RR \cup \{-\infty\}$ is defined by
\[
H(X) := - \trace X \log_2 X 
\]
for $X \in {\cal D}_n^+$ and $-\infty$ for $X \not \in {\cal D}_n^+$,
where $\log_2 X$ is the matrix logarithm defined by $X = 2^{\log_2 X}$.
For $X = k (\diag \lambda) k^{\dagger}$ with $k \in U_n$, $\lambda \in \RR^n$, it holds $H(X) = - \sum_{i=1}^n \lambda_i \log_2 \lambda_i$ if $\lambda \in \RR^n_{\geq 0}$ with $\sum_{i=1}^n \lambda_i =1$ and $-\infty$ otherwise, 
which is the {\em Shannon entropy} of the probability vector $\lambda$.
In particular, $-H$ is unitarily invariant convex.
The conjugate of the negative of the Shannon entropy is given by 
$x \mapsto \log_2 \sum_{i=1}^n 2^{x_i}$.
Therefore, the conjugate $(-H)^*$ of the negative of the von Neumann entropy is given by
\[
(-H)^*(Y) = \log_2 \trace 2^{Y} \quad (Y \in {\cal H}^*_n = {\cal H}_n).
\]
\end{Ex}

\subsection{Manifold ${\cal P}_n$ of positive definite matrices}\label{subsec:psd}

Let ${\cal P}_n \subseteq {\cal H}_n$ 
denote the set of positive definite $n \times n$ 
Hermitian matrices. It is an open set of ${\cal H}_n$ and 
has a natural manifold structure, where 
the tangent space $T_x$ at each point $x \in {\cal P}_n$ 
is naturally identified with ${\cal H}_n$.
The general linear group $GL_n(\CC)$ acts transitively on ${\cal P}_n$ 
as $(g,x) \mapsto g x g^{\dagger}$, where the isotropy subgroup at $I$ 
is the unitary group $U_n$.
The Riemannian metric $\langle,\rangle$  
on ${\cal P}_n$ is defined by $\langle X,Y \rangle := \trace x^{-1}X x^{-1}Y$ for $X,Y \in T_x$, which
which is invariant under this action. 
The resulting Riemannian manifold ({\em PD-manifold}) is a well-studied and representative example 
of Hadamard manifolds. 
Various Riemannian notions are written explicitly; see e.g., \cite[Chapter 6]{Bhatia_PositiveDefiniteMatrices} and \cite[Chapter II.10]{BridsonHaefliger1999}. 
For example, 
any geodesic is written as $t \mapsto g e^{t H} g^\dagger$ for $g \in GL_n(\CC)$ and 
$H \in {\cal H}_n$, where $e^{\bullet}$ denotes the matrix exponential.
In particular, a geodesic issuing at $x$ is written as 
$t \mapsto x^{1/2} e^{t x^{-1/2} H x^{-1/2}} x^{1/2}$ for $H \in {\cal H}_n$.
The parallel transport $\tau_{x \to I}:T_x \to T_I$ is given by
$\tau_{x \to I}H = x^{-1/2}Hx^{-1/2}$.
The asymptotic relation is computed via eigenvalue decomposition and QR-factorization as follows:
For $H \in T_x$, 
consider a geodesic ray $\gamma(t) = \exp_x t H 
= x^{1/2} e^{tx^{-1/2}Hx^{-1/2}}x^{1/2} = g e^{t \diag \lambda} g^{\dagger}$ for $g \in GL_n(\CC)$ and $\lambda \in \RR^n$ with $\lambda_1 \geq \lambda_2 \geq \cdots \geq \lambda_n$.
Observe\footnote{$d(b e^{t \diag \lambda}b^{\dagger}, e^{t \diag \lambda}) = 
d(y(t)y(t)^{\dagger}, I)
$ for $y(t) := e^{-t\diag \lambda/2}b e^{t \diag \lambda/2}$ that is bounded over $t \in \RR_{\geq 0}$.
}  
that $t \to b e^{t \diag \lambda} b^{\dagger}$ for upper-triangular $b$
is asymptotic to $t \mapsto e^{t \diag \lambda}$.
Via QR factorization $g = k b$ for unitary $k \in U_n$ and upper-triangular $b \in GL_n(\CC)$, 
geodesic $\gamma$ is asymptotic to the geodesic ray $t \mapsto e^{tk\diag \lambda k^{\dagger}}$
issuing at $I$.

The submanifold ${\cal P}_n^1 := \{ x \in {\cal P}_n \mid \det x = 1\}$ of ${\cal P}_n$
is totally geodesic in the sense that any geodesic in ${\cal P}_n$ containing any two points in ${\cal P}_n^1$ 
belongs to ${\cal P}_n^1$.
The special linear group $SL_n(\CC)$ acts isometrically and transitively on ${\cal P}^1_n$ 
as $(g,x) \mapsto g x g^{\dagger}$, where the isotropy subgroup at $I$ 
is the special unitary group $SU_n$. So ${\cal P}_n^1$ 
is the irreducible symmetric space $SL_n(\CC) / SU_n$, whereas 
${\cal P}_n = GL_n(\CC)/U_n$ is a reducible symmetric space 
admitting de Rham decomposition ${\cal P}_n = \RR \times {\cal P}_n^1$. 
%
%\[
%{\cal H}_n^0 := \{H \in {\cal H}_n \mid \trace H = 0\}.
%\]
%The cotangent space $T_I^*{\cal P}_n^1$ is identified with quotient space 
%${\cal H}_n/\ker \trace$. 
%We choose the representative 
%as the space of Hermitian matrices with trace one:
%\[
%{\cal D}_n := \{H \in {\cal H}_n \mid \trace H = 1\}.
%\]
%That is, $T_I^*{\cal P}_n^1 = {\cal D}_n$, 
%where the vector-space structure of ${\cal D}_n$ is induced from ${\cal H}_n/\ker \trace$.
The tangent space of ${\cal P}_n^1$ at $I$ is 
given by $T_I{\cal P}_n^1 = \{ H \in {\cal H}_n \mid \trace H = 0\}$.
Consider the holonomy group $Hol_I({\cal P}_n^1)$ at $I$.
It is known (see e.g., \cite[Chapter IV, Theorem 6.9, Remark 6.14]{Sakai1996}) that the holonomy group 
of a simply-connected irreducible symmetric space
is given by the adjoint action of the isotopy subgroup.
Therefore, $Hol_I({\cal P}_n^1)$ is equal to $SU_n$ acting on
$T_I {\cal P}_n^1$ as $(k,H) \mapsto k H k^{\dagger}$.
This is equal to the holonomy group $Hol_I({\cal P}_n)$ of ${\cal P}_n$ at $I$ 
since the holonomy group 
is the product of that of irreducible components.

\paragraph{Euclidean building $C{\cal P}_n^{\infty}$.}
It is well-known \cite[Chapter II.10]{BridsonHaefliger1999} that the boundary $C{\cal P}_n^{\infty}$ 
has the structure of {\em Euclidean building}, which is a representative example of non-manifold Hadamard spaces.
We give a quick and informal introduction of this fact.
A point in our building ${\cal B}_n$ is represented by a pair of
a complete flag $\{0\} < {\cal Y}_1 < {\cal Y}_2 < \cdots < {\cal Y}_n = \CC^n$ 
of vector subspaces and nonincreasing weight 
$\lambda \in \RR^n:\lambda_1 \geq \lambda_2 \geq \cdots \geq \lambda_n$, 
where $U \leq U'$ (resp. $U < U'$) means 
that vector space $U$ is a (resp. proper) subspace of vector space $U'$. 
Then ${\cal B}_n$ consists 
of (equivalence classes of) all such pairs, 
and is metrized so that for any ordered basis $u_1,u_2,\ldots,u_n$ of $\CC^n$, the map
\begin{equation}\label{eqn:expression}
\RR^n \ni \lambda \mapsto 
%\left\{ 
\begin{array}{c}
\lambda_{i_1} \geq \lambda_{i_2}\geq \cdots \geq \lambda_{i_n}, \\
\langle u_{i_1}\rangle < \langle u_{i_1},u_{i_2} \rangle < \cdots <  \langle u_{i_1},u_{i_2},\ldots,u_{i_n} \rangle 
\end{array}
%\right. 
\end{equation}
is an isometry from $\RR^n$ to ${\cal B}_n$,  
where all $i_1,i_2,\ldots,i_n$ are different, and
$\langle u_{i_1},u_{i_2},\ldots u_{i_k} \rangle$ denotes the vector subspace spanned by $u_{i_1},u_{i_2},\ldots,u_{i_k}$.
An {\em apartment} is a subset of ${\cal B}_n$ that is the image of such an isometry. 

The asymptotic class of a geodesic ray $t \mapsto g e^{t \diag \lambda}g^{\dagger}$ 
with nonincreasing $\lambda$ in ${\cal P}_n$
gives rise to such an expression 
$\lambda_1 \geq \lambda_2 \geq \cdots \geq \lambda_n$, 
$\langle u_{1} \rangle < \langle u_{1},u_{2} \rangle < \cdots <  \langle u_{1},u_{2},\ldots,u_{n}\rangle$,
where $u_i$ is the $i$-th column vector of $g$.
This expression is indeed independent of the choice of representatives, 
and provides an isometry from $C{\cal P}_n^{\infty}$ to ${\cal B}_n$. 
We can identify ${\cal H}_n (= T_I{\cal P}_n \simeq C{\cal P}_n^{\infty})$ 
with ${\cal B}_n$  
by associating $H \in {\cal H}_n$ 
with geodesic ray $t \mapsto e^{tH} = k e^{t \diag \lambda} k^{\dagger}$ with $k \in U_n$ and $\lambda \in \RR^n$.
We here note a link between unitarily invariant convexity  
and geodesic convexity. 
\begin{Lem}\label{lem:ui-convex=>g-convex}
In this identification ${\cal H}_n = {\cal B}_n$, 
any unitarily invariant convex function $f:{\cal H}_n \to \overline{\RR}$ 
is geodesically convex on ${\cal B}_n$. 
\end{Lem}
The proof is given in the appendix, since it only gives insights in this paper.
The building for $C({\cal P}_n^1)^{\infty}$ is obtained from the above expression by 
adding $\sum_{i=1}^n \lambda_i = 0$, 
where isometry (\ref{eqn:expression}) is from $\{ x\in \RR^n \mid \sum_{i=1}^n x_i = 0\}$.
It is an isometric subspace of ${\cal B}_n$.

\section{$Q$-gradient flow and asymptotic duality}\label{sec:Q-gradient_flow}

Let ${\cal M}$ be a Hadamard manifold, and 
let $Q:T^*{\cal M} \to \overline{\RR}$ be a function
on the cotangent bundle of ${\cal M}$. 
We are going to study the minimum $Q$-gradient-norm problem~(\ref{eqn:minimum_Q-gradient-norm_problem}).
Since the setting is too general, we begin to formulate a meaningful special class of $({\cal M},Q)$.
For $x \in {\cal M}$, let $Q_x := Q|_{T_x^*}$ 
denote the restriction of $Q$ to $T_x^*$.
The Legendre-Fenchel conjugate $Q^*:T{\cal M} \to \overline{\RR}$ of $Q$ is defined 
so that for each $x \in {\cal M}$, its restriction $(Q^*)_x$ to $T_{x}$ is equal to $Q_x^*$.
We then consider the following properties for $({\cal M},Q)$: 
\begin{itemize}
\item[Q1:] For each $x \in {\cal M}$, the restriction $Q_x$ is 
a $\RR_{\geq 0}$-valued convex function,  
$Q_x^2$ is continuously differentiable, 
and the level sets of $Q_x$ are bounded.
\item[Q2:] $Q$ is invariant under parallel transports; $Q(\tau_\gamma^1 p) = Q(p)$ 
for any $p \in T_x^*$, $x \in {\cal M}$ and curve $\gamma:[0,1] \to {\cal M}$ with $\gamma(0) = x$.
%Then the Legendre-Fenchel conjugate $Q^* = (Q_x^*)_{x \in M} :TM \to \overline{\RR}$ %defined by
%\[
%Q_x^*(u) := \sup_{p \in T_x^*} p(u) - Q_x(p) \quad (u \in T_x = T_x^{**})
%\]
%is also invariant under parallel transports.
%\footnote{}
%\begin{eqnarray*}
%G_{\gamma(t)}^*(\tau_\gamma^t u) &= & \sup_{p \in T_{\gamma(t)}^*} p(\tau_{\gamma}^tu) - G_{\gamma(t)}(p) \quad (u \in T_{\gamma(0)}) \\
%&=& \sup_{p \in T_{\gamma(t)}^*} (\tau_{\gamma}^{-t} p)(u) - G_{\gamma(0)}(\tau_\gamma^{-t}p) \\
%&=&  G_{\gamma(0)}^*(u).
%\end{eqnarray*}}
\item[Q3:] For $x,y \in {\cal M}$ and $u \in T_x$, $v \in T_y$, if 
$t \mapsto \exp_x tu$ and $t \mapsto \exp_y t v$ are asymptotic, then
$Q_x^*(-u) = Q_y^*(-v)$ holds. Define 
$Q_{-}^*:C{\cal M}^{\infty} \to \overline{\RR}$ by 
$Q_{-}^{*}(\xi) := Q^*_x(- u)$, where $t \mapsto \exp_{x}tu$ is the representative of $\xi \in C{\cal M}^{\infty}$ at $T_x$.
%\footnote{The negative $-\xi$ of $\xi \in C{\cal M}$ is not defined. 
%Even if $u \in T_x$ and $v \in T_y$ are the same point in $C{\cal M}^{\infty}$, $-u$ and $-v$ are not.
%}
%The corresponding function on $CM^{\infty}$ is denoted by $G^*: CM^{\infty} \to \overline{\RR}$.
\end{itemize}
Note that  
$Q^2_x$ is also convex in Q1.
%we will consider $Q^2/2:T^*{\cal M} \to \RR$ and its conjugate $(Q^2/2)^*: T{\cal M} \to \overline{\RR}$.  
If Q2 holds, then $Q^*$ is also invariant under any parallel transport, since
$Q_{\gamma(1)}^*(\tau_\gamma^1 u) =  \sup_{p \in T_{\gamma(1)}^*} p(\tau_{\gamma}^1u) - Q_{\gamma(1)}(p) = \sup_{p \in T_{\gamma(1)}^*} (\tau_{\gamma}^{-1} p)(u) - Q_{\gamma(0)}(\tau_\gamma^{-1}p) =  Q_{\gamma(0)}^*(u)$.
It is important to note that Q2 can be formulated in terms of the holonomy group:
\begin{itemize}
\item[Q2$'$:] For some {\em $Hol_{x_0}({\cal M})$-invariant function} $Q_0:T_{x_0}^* \to \overline{\RR}$, it holds
 \[
    Q(p) = Q_0(\tau_{x \to x_0} p) \quad (p \in T_x^*, x \in {\cal M}), 
    \]
    where $Hol_{x_0}({\cal M})$-invariance means that $Q_0$ satisfies
    \[
    Q_0(k\cdot p) = Q_0(p) \quad (\forall p \in T_{x_0}^*, \forall k \in Hol_{x_0}({\cal M})).
    \]
\end{itemize}
It is clear that Q2 implies Q2$'$.
Conversely, If $Q$ satisfies Q2$'$, then,
for $\gamma:[0,1] \to {\cal M}$ with $\gamma(0) = x$ and $\gamma(1) = y$.
we have
    \[
    Q(\tau_\gamma^{1} p ) = Q_0(\tau_{y \to x_0} \circ \tau_\gamma^{1} p) = Q_0(\tau_{y \to x_0} \circ \tau_\gamma^{1} \circ \tau_{x_0 \to x} \circ \tau_{x \to x_0}  p) = Q_0(\tau_{x \to x_0} p) =Q(p). 
    \]

\begin{Ex}[Euclidean space]
    Suppose that ${\cal M} = \RR^n$. 
    Then $T_x$, $T_x^*$, and $C{\cal M}^{\infty}$ are identified with $\RR^n$, 
    where all parallel transports are the identity map.  
    Any convex function 
    $Q:\RR^n \to \RR_{\geq 0}$ for which it has bounded level sets and $Q^2$ is differentiable 
    gives rise to $Q:T^*{\cal M} \to \RR_{\geq 0}$ satisfying Q1-3. 
\end{Ex}

\begin{Ex}[Dual norm]
    The dual norm $Q := \|\cdot\|^*:T^*{\cal M} \to \RR_{\geq 0}$ 
    satisfies Q1-3. 
    Indeed, the conjugate is given by $Q^*_x(u) = \delta_{B_{\|\cdot\|}(0;1)}$. 
    The norm $\|\cdot\|$ is invariant under parallel transports, 
    and gives the speed $\|H\|$ of the geodesic ray $t \mapsto \exp_x t H$ if 
    $Q^*$ is viewed as a function on $C{\cal M}^{\infty}$.
    From this, we see Q2 and Q3. 
    More generally, 
    consider any nondecreasing convex function $h:\RR_{\geq 0} \to \RR_{\geq 0}$, and let $Q:= h(\|\cdot\|^*)$. Then it holds $Q^* = h^*(\|\cdot\|)$ for the conjugate $h^*(y) := \sup_{t \geq 0} yt - h(t)$ that is also nondecreasing.
    From this, we see Q1-3. 
    If the holonomy group acts transitively on the sphere, 
    then $Q$ is necessarily written in this way.
\end{Ex}

\begin{Ex}[PD-manifold and unitarily invariant convex function]\label{ex:PD-manifold}
Consider the PD-manifold ${\cal P}_n$. 
Choose a unitarily invariant convex function 
$Q_0:{\cal H}_n \to \RR_{\geq 0}$ such that $Q_0^2$ is continuously differentiable, and regard it as $T_I^* \to \RR_{\geq 0}$. 
Then it is a $Hol_I({\cal P}_n)$-invariant function.
Define $Q:T^* {\cal P}_n \to \RR$ by 
\[
Q_x(Y) := Q_0(\tau_{x \to I} Y) \quad (Y \in T^*_x{\cal P}_n). 
\]
Then, $Q$ satisfies Q1 and Q2.
To see Q3, choose $H \in T_x$ and 
suppose that $x^{-1/2}Hx^{-1/2} = u (\diag \lambda) u^{\dagger}$ for $u \in U_n$ and $\lambda \in \RR^n$.
Then $Q_x^*(- H) = Q_0^*(- \diag \lambda)$. On the other hand, 
as explained in Section~\ref{subsec:psd}, 
geodesic $t \mapsto \exp_x t H$ 
is asymptotic to $t \mapsto \exp_I t k (\diag \lambda) k^{\dagger}$ for some $k \in U_n$.
Namely, $H \in T_x$ and $k (\diag \lambda) k^{\dagger} \in T_I$ represent the same point in $C{\cal M}^{\infty}$.
Then $Q^*_x(-H) = Q^*_0 (- \diag \lambda) = Q^*_I(- k (\diag \lambda) k^{\dagger})$. 
This implies Q3.
This construction is generalized to the product of PD-manifolds, 
more generally, symmetric spaces of nonpositive curvature.
\end{Ex}
As seen in these examples, the following generally holds.
\begin{Prop}\label{prop:Q2=>Q3} 
{\rm Q2} implies {\rm Q3}.
\end{Prop}
The proof relies on the Berge holonomy theorem and properties of symmetric spaces, 
given in the appendix.

Let $f: {\cal M} \to \RR$ be a 
twice differentiable geodesically convex function.
Consider the minimum $Q$-gradient-norm problem (\ref{eqn:minimum_Q-gradient-norm_problem}) 
for ${\cal M},Q,f$.
%\begin{equation}\label{eqn:problem}
%{\rm Min.} \quad Q(df_x) \quad {\rm s.t.} \quad x \in {\cal M}.  
%\end{equation}
Our goal is to establish an extension of (\ref{eqn:min_gradient-norm})
to this setting. We start with the weak duality that generalizes \cite[Lemma 2.2]{HiraiSakabe2024FOCS}.
\begin{Lem}[Weak duality]\label{lem:weakduality}
Suppose that $Q$ satisfies {\rm Q3}. Then it holds
\begin{equation}\label{eqn:weakduality}
\inf_{x \in {\cal M}} Q (df_x) \geq \sup_{\xi \in C{\cal M}^{\infty}} - f^{\infty}(\xi) - Q_{-}^*(\xi). 
\end{equation}
\end{Lem}
\begin{proof} From the definition of $f^{\infty} = f_x^{\infty}$, it holds
\begin{eqnarray*}
f_x^{\infty}(\xi) &= &\lim_{t \to \infty} \frac{f(\exp_x t\xi) - f(x)}{t} \geq \lim_{t \to 0} \frac{f(\exp_x t\xi) - f(x)}{t} \\
& = &  df_x(\xi) = - df_x(- \xi) \geq - Q_x(df_x) - Q_x^*(-\xi), 
\end{eqnarray*}
where the first inequality follows from convexity of $f$ and the second inequality
follows from the Fenchel-Young inequality
$Q_x(p) + Q_x^*(- \xi) \geq p(-\xi)$
for $p= df_x \in T_x^*$ and $- \xi \in T_x$.
By Q3, we have the claim.
\end{proof}

%\begin{Question}
%Does the equality (strong duality) hold in (\ref{eqn:weakduality}) ?
%\end{Question}
In the case of Euclidean space, the strong duality holds.
\begin{Ex}
Consider ${\cal M} = \RR^n$ and a lower semicontinuous convex function $Q:\RR^n \to \overline{\RR}$. 
Then $C{\cal M}^{\infty}$ is identified with $\RR^n$, and
$f^{\infty}$ is the support function of the closure $C$ of $\{ df_x  \mid x \in \RR^n\} \subseteq (\RR^n)^*$ 
(see e.g.,\cite[Section 3.3]{HiraiSakabe2024FOCS}).
From the Fenchel duality~\cite[Section 31]{Rockafellar}, it holds
%\begin{eqnarray*}
\[
    \inf_{x \in \RR^n} Q(df_x) = \inf_{p \in C} Q(p) = \inf_{p \in (\RR^n)^*} Q(p) + \delta_{C}(p)     = \sup_{\xi \in \RR^n} - Q^*(- \xi) - f^{\infty}(\xi).
%\end{eqnarray*}
\]
\end{Ex}
%\begin{Ex}
%Consider the case $G=\|\cdot\|^*$; then $G^* = \delta_{B(0;1)}$ (indicator function of unit ball %$B(0;1)$).
%Hirai and Sakabe~\cite{HiraiSakabe2024FOCS} established 
%the strong duality by gradient flow, 
%which motivates the subsequent argument.
%\end{Ex}
We next introduce a generalization of the gradient flow $x(t)$ that  is expected to
attain the equality in (\ref{eqn:weakduality}) by the limit.
From now on, we assume that $({\cal M}, Q)$ satisfies Q1, Q2, and Q3.
\begin{Def}[$Q$-gradient flow]\label{def:Q-gradient}
The {\em $Q$-gradient} $\nabla^{Q} f(x)$ at $x \in {\cal M}$ is defined by
\begin{equation*}
\nabla^{Q} f(x) := \frac{1}{2}d Q^{2}_{df_x}\mid_{T_x^*}\quad  (\in T_x^{**} = T_x).  
\end{equation*}
The {\em $Q$-gradient flow} is the solution $x(t)$ of the initial value ODE:
\begin{equation}\label{eqn:Q-gradient_flow}
\dot x(t) = - \nabla^{Q} f(x(t)), \quad x(0) = x_0.
\end{equation}
\end{Def}

\begin{Ex}
Suppose that $Q = \|\cdot\|^*$. From $Q(df_x) = \|df_x\|^* = \|\nabla f(x)\|$, 
we see that the $Q$-gradient and the ordinary gradient are the same:
$\nabla^{Q} f(x) = \nabla f(x)$.
Accordingly, the $Q$-gradient flow is the ordinary gradient flow of $f$. 
%More generally, suppose that $Q$ is a Minkowski norm in Example~\ref{ex:Finsler},
%Then, we see from $Q^2_x/2 = (N_x^2/2)^*$ that $\nabla^Q f(x)$ is equal to 
%the gradient of $f(x)$ in this Finsler manifold $({\cal M}, N)$.
\end{Ex}
%\begin{eqnarray*}
%\frac{1}{2} dG^2_{df_x} \mid_{T^*_x} &=& \nabla f(x),\\
%G(df_x) &=& \|\nabla f(x)\|.
%\end{eqnarray*}
% \begin{proof}
% It is clear from convex analysis that $Q_x^* = (\|\cdot \|^*)^* = \delta_{B(0;1)}$.
% %For $v \in T_x \setminus \{0\}$, 
% %choose $p^* \in T_x^*$ so that $\|p^*\|^* = 1$ and
% %$p^*v = \|v\|$.
% %Then $G^*(v) = \sup_{p \in T_x^*} vp - h(\|p\|) = \sup_{t\geq 0} t\|v\| - h(t) = h^*(\|v\|)$.
% %
% Define linear isomorphism $\varphi:T_x^* \to T_x$ via $p(v) = \langle \varphi(p),v \rangle$ $(v\in T_x)$. 
% Then, $\varphi(df_x) = \nabla f(x)$. From $\|p\|^* = \|\varphi(p)\|$, 
% we have $Q(df_x) = \|df_x\|^* = \|\nabla f(x)\|$.
% %From $\|p\|^* = \max_{v\in T_x:\|v\|=1} p(v) = \max_{v\in T_x:\|v\|=1} \langle \varphi(p),v %\rangle = \|\varphi(p)\|$.
% %Then we have $G(df_x) = h(\|df_x\|^*) = h(\|\varphi(df_x)\|) = h(\|\nabla f(x)\|)$, and
% For $p \in T_x^*$, we have
% \begin{eqnarray*}
% (\nabla^{Q}f(x))(p) & =& \frac{d}{dt}\mid_{t=0} \frac{(\|df_x+tp\|^*)^2}{2} =  \frac{d}{dt}\mid_{t=0} \frac{\langle \varphi(df_x)+t\varphi(p),\varphi(df_x)+t\varphi(p)\rangle^2}{2} \\
% & = & \langle \varphi(df_x), \varphi(p) \rangle = (\nabla f(x)) (p).
% \end{eqnarray*}
% \end{proof}

Consider the $Q$-gradient flow $x(t)$ in (\ref{eqn:Q-gradient_flow}). 
By Q1 and the Cauchy-Peano existence theorem for ODE, 
such an $x(t)$ indeed exists and is defined on $[0,T)$ for some $T> 0$.
The following property generalizes the well-known monotonicity property of
$t \mapsto \|\nabla f(x(t))\|$ for the ordinary gradient flow $x(t)$.
\begin{Lem}\label{lem:nonincreasing}
$t \mapsto Q(df_{x(t)})$ is monotonically nonincreasing.
\end{Lem}
\begin{proof}
We show $(d/dt) Q^{2}(df_{x(t)}) \leq 0$. Let $\tau^{-{\delta}}$ denote the parallel transport along the curve $x(\cdot)$ from $t+\delta$ to $t$. Then it holds 
\[
 \frac{d}{d\delta}\mid_{\delta = 0} Q^{2}(df_{x(t+ \delta)}) = \frac{d}{d\delta}\mid_{\delta = 0} Q^2(\tau^{- \delta } df_{x(t+\delta )})  = d Q^2_{df_{x(t)}}( \nabla_{\dot x(t)} df) = - 2 \nabla df_{x(t)}(\dot x(t),\dot x(t)) \leq 0,
\]
%\begin{eqnarray*}
%&& \frac{d}{d\delta}\mid_{\delta = 0} Q^{2}(df_{x(t+ \delta)}) = \frac{d}{d\delta}\mid_{\delta = 0} Q^2(\tau^{- \delta } df_{x(t+\delta )}) \\
%&& = d Q^2_{df_{x(t)}}( \nabla_{\dot x(t)} df) = - 2 \nabla df_{x(t)}(\dot x(t),\dot x(t)) \leq 0,
%\end{eqnarray*}
where the first equality follows from Q2,  
the second from the chain rule and the definition of covariant derivative,  
the third from $d Q_{df_{x(t)}}^2 = - 2\dot x(t)$, 
and the last inequality from convexity of~$f$.
\end{proof}

\begin{Lem}\label{lem:upperbound}
For a constant  $\alpha \geq 0$ it holds 
$
d(x(s),x(t)) \leq \alpha |s-t|. 
$ 
\end{Lem}

\begin{proof}
The velocity of $x(t)$ is given by
\[
\|\dot x(t)\| =  \| d(Q^2_{df_{x(t)}}/2)\mid_{T_x^*}\| = \| d(Q^2_{0}/2)_{\tau_{x(t) \to x_0}df_{x(t)}}\|.
\] 
By the previous lemma, $t \mapsto Q(df_{x(t)}) = Q_0(\tau_{x(t) \to x_0} df_{x(t)})$ is monotonically nonincreasing.
Therefore, $\tau_{x(t) \to x_0} df_{x(t)}$ belongs to the level set 
$\lv(Q_0, Q_0(df_{x_0}))$. Since it is compact by Q1,
the maximum $\alpha$ of $p \mapsto \| d(Q^2_0/2)_p\|$ over $p \in \lv(Q_0, Q_0(df_{x_0}))$ exists.
Then we have
$
\|\dot x(t)\| \leq  \alpha$.
By the triangle inequality, we have
$d(x(s), x(t)) \leq \int_{s}^t \|\dot x(r)\|dr \leq \alpha |s-t|.$
\end{proof}

\begin{Lem}\label{lem:completeness}
The vector field $(- \nabla^{Q}f(x))_{x\in {\cal M}}$ is complete.
\end{Lem}
\begin{proof}
It suffices to show that 
$\lim_{t \to T} x(t)$ exists in ${\cal M}$ if $T < \infty$. 
By the previous lemma, 
for every sequence $(t_i)_{i=1,2,\ldots}$ in  $[0,T)$ converging to $T$, 
the sequence $(x(t_i))_{i=1,2,\ldots}$ in  ${\cal M}$ is Cauchy.
By the completeness of ${\cal M}$, the limit of the sequence exists, which must be $\lim_{t \to T} x(t) := x^*$. 
%Now the $Q$-gradient flow $x(t)$ $(t \in [0,T))$ is connected with the $Q$-gradient flow 
%$y(t)$ $(t \in [0, T'))$ with $y(0) = x^*$ and $T' > 0$.
%Indeed, by the previous lemma, $[0,T) \in t \mapsto x(t)$ is continuous, and $(x(t))_{t\in[0,T)}$ belongs 
%to a closed ball with finite radius $\alpha T$. 
%Otherwise, there is a sequence $(t_i)_{i=1,2,\ldots} \subseteq [0,T)$ 
%such that $t_i \to T$ and $(x(t_i)){i=1,2,\ldots}$ does not converge.
%By completeness of ${\cal M}$, $(x(t_i)){i=1,2,\ldots}$ is not Cauchy.
%However this contradicts the previous lemma $d(x(t_i), x(t_{j})) \leq \alpha |t_i - t_j|$.
% and the curve $x(t) (t \in [0,T))$ belongs to the closed ball $B$ with center $x_0$ and radius $\alpha T$.
% Since $B$ is compact,  
% $x(t)$ $(t \in [0,T))$ has accumulation points in $B$.
% Suppose indirectly that distinct accumulation points $y,y'$ exist. 
% For $\epsilon := d(y,y')/3 > 0$, consider disjoint balls $B_1 = B(y;\epsilon)$, $B_2 = B(y';\epsilon)$.
% The curve $x(t)$ $(t \in [0,T))$ meets them infinitely many times.
% The distance between $B_1$ and $B_2$ is $\epsilon >0$.
% This implies that the length of curve $x(t)$ $(t \in [0,T))$ must be $\infty$, 
% which contradicts $\int_{0}^T \|\dot x(s)\|ds \leq Q_0(df_{x_0}) \alpha T$.
% Thus an accumulation point is unique and must be the limit $\lim_{t \to T} x(t)$.
\end{proof}
Thus, $x(t)$ is defined on $\RR_{\geq 0}$.
Next, we show the energy identity of the $Q$-gradient flow. 
\begin{Lem}[Energy identity]\label{lem:energy_identity}
It holds
\begin{equation}\label{eqn:energy_identity}
f(x(t)) - f(x_0) =   -\int_{0}^t \frac{1}{2} Q^2(df_{x(s)}) + \left(\frac{1}{2}Q^2\right)^* (- \dot x(s)) ds. 
%&=& -\int_{0}^T  G^2_x(- df_x) + G_x(-df_x) G_x^* \left(\frac{\dot x}{G_x(-df_x)}\right) dt 
\end{equation}
If $Q(df_x) > 0$ for every $x \in {\cal M}$, then
\begin{equation}\label{eqn:energy_identity_modified}
  f(x(t)) - f(x_0) + \int_0^{t} Q^* \left( -\frac{\dot x(s)}{Q(df_{x(s)})} \right) Q(df_{x(s)}) ds = - \int_0^t Q^2(df_{x(s)})ds. 
\end{equation}
%\end{eqnarray*}
\end{Lem}
\begin{proof}
If $v =d(Q_x^2/2)_p$ for $p \in T_x^*,v \in T_x$, then by (\ref{eqn:subdifferential}) it holds
\begin{equation}\label{eqn:p(v)}
p(v) = \frac{1}{2}Q_x^2(p) + \left( \frac{1}{2}Q^2_x\right)^*(v).
\end{equation}
Therefore, for the setting $v := - \dot x(t)$ and $p := df_{x(t)}$, 
it holds $v = d(Q_x^2/2)_{p}$ by Definition~\ref{def:Q-gradient}. 
From (\ref{eqn:p(v)}), we have
\begin{equation*}
\frac{d}{dt}f(x(t)) =  df_{x(t)}(\dot x(t)) =  -\frac{1}{2} Q^2(df_{x(t)}) - \left(\frac{1}{2}Q^2\right)^* (-\dot x(t)).
\end{equation*}
By integrating it from $0$ to $t$, we have 
(\ref{eqn:energy_identity}).

We show the latter part.
If $v = d(Q_x^2/2)_p \in Q_x(p) \cdot \partial (Q_x)(p)$ and $Q_p(p) > 0$,
%\footnote{$Q_x (= \sqrt{Q_x^2})$ may be nondifferentiable at the origin}, 
then 
it holds $v/Q_x(p) \in  \partial (Q_x) (p)$, and 
$
 p(v/Q_x(p)) = Q_x(p)  + Q_x^*(v/Q_x(p))$ by (\ref{eqn:subdifferential}), 
that is,
\[
p(v) = Q^2_x(p) + Q_x(p) Q_x^{*}\left(\frac{v}{Q_x(p)}\right).
\]
Considering the setting $p := df_x$ 
and $v := -\dot x(t)$, we have (\ref{eqn:energy_identity_modified}) as above.
\end{proof}
% \begin{Lem}
%  If $\displaystyle \inf_{x \in {\cal M}} Q(df_x) > 0$, 
%  then $\displaystyle \lim_{t \to \infty}f(x(t)) = - \infty$ and $\displaystyle \lim_{t \to \infty} d(x_0,x(t)) = \infty$.
% \end{Lem}
% \begin{proof}
% Note that $(Q^2/2)^{*}$ is nonnegative-valued 
% since $(Q^2_x/2)^{*}(u) \geq u(0) - (Q^2/2)(0) = 0$.
% Let $v^* := \inf_{x \in {\cal M}} Q(df_{x(t)}) > 0$,
% By Lemmas~\ref{lem:nonincreasing} and \ref{lem:energy_identity}, it holds
% $f(x(T)) - f(x_0) \leq - v^*T$. 
% Thus $\lim_{t \to \infty}f(x(t)) = - \infty$.

% Suppose for contradiction that  $d(x_0,x(t))$ has an accumulation point 
% $r < \infty$.  Then 
% $x(t)$ has an accumulation point 
% in the closed ball $B$ with center $x_0$ and radius $r$. 
% $\inf_{x \in B} f(x) = \lim_{t \to \infty} f(x(t))= -\infty$, contradicting $f: M \to \RR$.
% \end{proof}

% We note that if $Q(0) = 0$, 
% then the $Q$-gradient flow also does not increase $f$, 
% though we do not use this fact. 
% \begin{Lem}
% If $Q_x(0) = 0$ for $x \in {\cal M}$, 
% then $t \mapsto f(x(t))$ is nonincreasing.
% %In addition, if $\inf_{x \in {\cal M}}Q(df_x) > 0$, 
% %then $\lim_{t \to \infty} f(x(t)) = -\infty$.
% \end{Lem}
% \begin{proof}
% This follows from energy identity and nonnegativity of $(Q^2/2)^*$
% by $(Q_x^2/2)^*(v) \geq v(0) - (Q_x/2)(0) = 0$.
% \end{proof}

For $t \in \RR_{\geq 0}$, define $R(t) \in \RR_{\geq 0}$ and $u(t) \in T_{x_0}$ by
    \begin{equation*}
     R(t) := \int_0^t Q(df_{x(s)})ds, \quad x(t) = \exp_{x_0} R(t) u(t).
    \end{equation*}
\begin{Lem}
Suppose that $\inf_{x\in {\cal M}} Q(df_{x}) > 0$. Then the following hold:
\begin{itemize}
\item[(1)] $(u(t))_{t \in \RR_{\geq 0}}$ is bounded in $T_{x_0}$.
\item[(2)] For $t \in \RR_{>0}$ and $r \in [0, R(t)]$, it holds
\begin{equation}\label{eqn:important}
\frac{f(\exp_{x_0} r u(t)) - f(x_0)}{r} +  \frac{1}{R(t)} 
\int_0^{t} Q^* \left( -\frac{\dot x(s)}{Q(df_{x(s)})} \right) Q(df_{x(s)}) ds 
\leq - \lim_{t \to \infty} Q(df_{x(t)}).
\end{equation}
\end{itemize}
\end{Lem}
The essence of the inequality (\ref{eqn:important}) is given in the proof of \cite[Theorem 3.1]{HiraiSakabe2024FOCS}.
\begin{proof}
(1). 
By $d(x_0,x(t)) = \|R(t) u(t)\| = R(t)\|u(t)\|$ and
Lemmas~\ref{lem:nonincreasing} and \ref{lem:upperbound}, it holds 
\[
\|u(t)\| = \frac{d(x_0,x(t))}{R(t)} = \frac{1}{\frac{1}{t}\int_{0}^t Q(df_{x(s)})ds}\frac{d(x_0,x(t))}{t} 
\leq \frac{\alpha}{\lim_{t \to \infty} Q(df_{x(t)})} < \infty.  
\]
%Thus $u(t)$ is bounded. 

(2). Consider the geodesic $[0,R(t)] \ni r \mapsto \exp_{x_0} r u(t)$ from $x_0$ to $x(t)$.
By convexity of $f$ along this geodesic, it holds
\begin{equation*}
f(\exp_{x_0} r u(t)) - f(x_0) \leq \frac{r}{R(t)}(f(x(t))- f(x_{0})) \quad (r \in [0,R(t)]).
\end{equation*}
By the energy identity (\ref{eqn:energy_identity_modified}), we have
\begin{equation}\label{eqn:important0}
\frac{f(\exp_{x_0} r u(t)) - f(x_0)}{r} \leq \frac{-1}{R(t)} \left( 
\int_0^{t} Q^* \left( -\frac{\dot x(s)}{Q(df_{x(s)})} \right) Q(df_x(s)) ds + \int_0^t Q^2(df_{x(s)})ds
\right).
\end{equation}
By the Cauchy-Schwarz inequality and Lemma~\ref{lem:nonincreasing}, it holds
\[
- \frac{\int_0^t Q^2(df_{x(s)})ds}{\int_0^t Q(df_{x(s)})ds} \leq -\frac{1}{t}\int_0^t Q(df_{x(s)})ds \leq - \lim_{t \to \infty} Q(df_{x(t)}).
\]
By substituting it for the second term of the RHS in (\ref{eqn:important0}),
we have (\ref{eqn:important}).
\end{proof}
The main result in this section is an extension of (\ref{eqn:min_gradient-norm})
to a basic class of ${\cal M}$:
\begin{Thm}\label{thm:main}
Assume that ${\cal M} = {\cal M}_1 \times {\cal M}_2 \times \cdots \times {\cal M}_d$ where each ${\cal M}_i$ satisfies one of the following
\begin{itemize}
\item[(i)] ${\cal M}_i = \RR^n$. 
\item[(ii)] The holonomy group of ${\cal M}_i$ acts transitively on the sphere.
\item[(iii)] ${\cal M}_i = {\cal P}^1_n$.
%\item[(iv)] ${\cal M}$ is the product of manifolds satisfying above (i),(ii), or (iii). 
\end{itemize}
Suppose $\inf_{x \in {\cal M}} Q(df_x) >0$. Then it holds
    \begin{equation}\label{eqn:main}
  \lim_{t \to \infty}  Q(df_{x(t)}) = \inf_{x \in {\cal M}} Q(df_x) = \sup_{\xi \in C{\cal M}^{\infty}} - f^{\infty}(\xi) - Q^*_{-}(\xi)= - f^{\infty}(u^*) - Q^*_{-}(u^*),
\end{equation}
where $u^*$ is any accumulation point of $(u(t))_{t \in \RR_{\geq 0}}$. 
%If the minimizer $\xi^*$ 
%of $\xi \mapsto f^{\infty}(\xi)+ Q^*_{-}(\xi)$ 
%is unique, then $u(t)$ converges to $\xi^*$. 
\end{Thm}
From the proof for (ii), the statement of the theorem holds when $Q$ is written as $Q= h(\|\cdot\|^*)$ (independent of the manifold). 
The proof for (iii) works for the real PD-manifold ${\cal P}^1_{n}(\RR) := {\cal P}^1_n \cap GL_n(\RR) \simeq SL_n(\RR)/SO(n)$.
In the light of the Berge holonomy theorem, 
the remaining case is irreducible symmetric spaces with rank $\geq 2$, other than ${\cal P}^1_{n}$ and ${\cal P}^1_{n}(\RR)$.

The proof follows the approach developed in \cite{Hirai2025,HiraiSakabe2024FOCS}, 
which is a combination of the weak duality (Lemma~\ref{lem:weakduality}) and the inequality (\ref{eqn:important}).
In the case of $Q = \|\cdot\|^*$ (with $Q^*_{-} = \delta_{B_{\|\cdot\|}(0,1)}$), 
the second term of the LHS of (\ref{eqn:important}) vanishes, and 
$u(t)$ belongs to $B_{\|\cdot\|}(0,1)$ by the triangle inequality $d(x_0,x(t)) \leq R(t)$. 
From $t \to \infty$ and $r \to \infty$, we obtain $u^* \in B_{\|\cdot\|}(0,1)$ 
with the desired inequality $f^{\infty}(u^*) \leq - \inf_{x \in {\cal M}} Q(df_x)$ 
(the reverse of the weak duality (Lemma~\ref{lem:weakduality})).
For the general case, however, 
the second term in (\ref{eqn:important}) makes the argument more involved, and 
we were only able to prove the above incomplete result. 
Fortunately, the case (iii) is expected to have enough applications, 
as we will see in Section~\ref{sec:entanglement_polytope}. 
The proof incorporates further nontrivial combinations of convexity (Jensen's inequality),
the Rauch comparison theorem for (ii), and matrix analysis for (iii), 
which may be interesting its own right.

\begin{proof}[Proof of Theorem~\ref{thm:main}]
% Case (i). $Q_x^* = \delta_{C_x}$ for a closed convex set $C_x$ in $T_x$.
% Necessarily  $- \dot x(s)/Q(df_{x(s)}) \in C_x$ and  
% the second term of LHS of (\ref{eqn:important}) is zero. 

Case (i). Suppose that ${\cal M} =\RR^n$. From $x(t) -x_0 = R(t)u(t)$, we have
\[
u(t) = \frac{1}{R(t)} \int_{0}^t \dot x(s) ds.
\]
Also, by convexity (Jensen's inequality) of $Q^*$, we have
\[
\frac{1}{R(t)} 
\int_0^{t} Q^* \left( -\frac{\dot x(s)}{Q(df_{x(s)})} \right) Q(df_{x(s)}) ds
\geq Q^* \left( \frac{1}{R(t)}\int_{0}^t - \dot x(s) ds \right) = Q^*(-u(t)).
\]
Therefore, from (\ref{eqn:important}), we have
\begin{equation}\label{eqn:becomes}
\frac{f(\exp_{x_0} r u(t)) - f(x_0)}{r} +  Q^*(-u(t)) 
\leq - \lim_{t \to \infty} Q(df_{x(t)}).
\end{equation}
Choose any subsequence $(t_i)_{i=1,2,\ldots}$ such that 
$t_i \to \infty$ and $u(t_i)$ converges to some $u^* \in T_{x_0}$.  
Since $\inf_{x \in {\cal M}} Q(df_x) >0$, it holds $R(t_i) \to \infty$.
By (\ref{eqn:becomes}), for any $r \in \RR_{\geq 0}$ it holds
\[
\frac{f(\exp_{x_0} r u^*) - f(x_0)}{r} +  Q^*(-u^*) 
\leq - \lim_{t \to \infty} Q(df_{x(t)}).
\]
For $r \to \infty$, we obtain 
\begin{equation}
f^{\infty}(u^*) + Q^*(-u^*) \leq - \lim_{t \to \infty}Q(df_{x(t)}).
\end{equation}
Combining it with the weak duality (Lemma~\ref{lem:weakduality}), we have
\begin{eqnarray*}
\lim_{t \to \infty}Q(df_{x(t)}) &\leq& - f^{\infty}(u^*) - Q^* (-u^*)  \\ 
& \leq& \sup_{\xi \in C{\cal M}^{\infty}} - f^{\infty}(\xi) - Q^*_{-}(\xi)
\leq \inf_{x \in {\cal M}} Q(df_x) \leq \lim_{t \to \infty}Q(df_{x(t)}). 
\end{eqnarray*}
Thus, the equality (\ref{eqn:main}) holds.
%Further, if the minimizer $\xi^*$ is unique, then $u^*$ must be $\xi^*$. 

%\end{proof}

% We study a more detailed expression of (\ref{eqn:u(t)}).
% \begin{Lem}[{See e.g.,\cite[p. 35]{Sakai1996}}]
% For $u \in S_{x_0}$, $t \in \RR_{++}$, and $H \in T_{tu}T_{x_0} = T_{x_0}$, it holds
%     \[
% d (\exp_{x_0})_{t u} H  = \frac{X(t)}{t},   
%     \]
%     where $X(t)$ is the Jacobi field along geodesic $t \mapsto \exp_{x_0}tu$ 
%     with $X(0) = 0$ and $\nabla X(0) = H$. 
% \end{Lem}

%\begin{proof}
Case (ii). Suppose that $Hol_{x_0}({\cal M})$ acts transitively on the sphere. Let $H(t) := R(t)u(t)$.
Then 
\[
u(t) = \frac{1}{R(t)}\int_{0}^t \dot H(s) ds.
\]
By
$x(s) = \exp_{x_0} H(s)$, it holds
\[
\dot x(s) = d(\exp_{x_0})_{H(s)} \dot H(s).
\]
It is well-known that the differential of the exponential map $\exp_{x_0}$
is given by a Jacobi field; see e.g.,~\cite[p. 35]{Sakai1996}.
Specifically, it holds 
\[
d (\exp_{x_0})_{r H(s)} \dot H(s)  = \frac{J_s(r)}{r},   
\]
for the Jacobi field $J_s(r)$ along geodesic $r \mapsto \gamma(r) = \exp_{x_0}r H(s)$ 
with $J_s(0) = 0$ and $\nabla_{\dot \gamma(0)} J_s = \dot H(s)$.
Therefore, 
$\dot x(s) = J_s(1).$
By the Rauch comparison theorem~\cite[Theorem IV. 2.3]{Sakai1996} between $\RR^n$ and ${\cal M}$, it holds
$
\|J_s(r)\| \geq \|r \dot H(s)\|
$, and hence
\[
\|\dot x(s)\| \geq \|\dot H(s)\|.
\]
Since $Q^*$ is written as $Q^* = h^*(\|\cdot\|)$ 
for a nondecreasing convex function $h^*: \RR_{\geq 0} \to \RR$, we have
\[
 Q^* \left( -\frac{\dot x(s)}{Q(df_{x(s)})} \right) = h^*\left( \frac{\|\dot x(s)\|}{Q(df_{x(s)})} \right) \geq  h^*\left( \frac{\|\dot H(s)\|}{Q(df_{x(s)})} \right) =  Q^* \left( -\frac{\dot H(s)}{Q(df_{x(s)})} \right),
\]
and
\begin{eqnarray*}
&& \frac{1}{R(t)} 
\int_0^{t} Q^* \left( -\frac{\dot x(s)}{Q(df_{x(s)})} \right) Q(df_{x(s)}) ds
\geq \frac{1}{R(t)} 
\int_0^{t} Q^* \left( -\frac{\dot H(s)}{Q(df_{x(s)})} \right) Q(df_{x(s)}) ds \\
&& \geq Q^* \left( \frac{1}{R(t)}\int_{0}^t - \dot H(s) ds \right) = Q^*(- u(t)),
\end{eqnarray*}
where the second inequality is Jensen's inequality as above.
Therefore, we have (\ref{eqn:becomes}). 
The rest is the same as in the case (i).

Case (iii). Suppose that ${\cal M} = {\cal P}_n^1$. Since 
the $SL_n(\CC)$-action on ${\cal P}_n^1$ is transitive and naturally extends to $C({\cal P}_n^1)^{\infty}$,
we can assume that $x_0 = I$.  
Specifically, choose $g\in SL_n(\CC)$ with $gg^{\dagger}=x_0$, 
and consider $\tilde f(x) := f(g x g^{\dagger})$ 
and $\tilde Q(Y) := Q((g^{\dagger})^{-1} Y g^{-1})$ instead of $f,Q$.

As in (ii), let $H(t) := R(t)u(t)$. Then $x(t) = e^{H(t)}$ and
\[
u(t) = \frac{1}{R(t)} \int_{0}^t \dot H(s) ds.
\]
\begin{Clm}\label{clm:show}
There are $k_1,k_2,\ldots,k_m \in SU_n$ and $p_1,p_2,\ldots,p_m \in \RR_{>0}$ 
with $\sum_{j=1}^m p_j =1$ such that 
\begin{equation*}
\dot H(s) = \sum_{j=1}^m p_j k_j^{\dagger} (\tau_{x(s) \to x_0} \dot x(s)) k_j. 
\end{equation*}
%
%\begin{equation}\label{eqn:show}
%Q^* \left( -\frac{x(s)^{-1/2} \dot x(s) x(s)^{-1/2}}{Q(df_{x(s)})} \right)  \geq Q^*\left( -%\frac{\dot H(s)}{Q(df_{x(s)})}  \right).
%\end{equation}    
\end{Clm}
We show this claim later.
By this claim, we obtain
\begin{equation}\label{eqn:IEMI}
 Q^* \left( -\frac{\dot x(s)}{Q(df_{x(s)})} \right) \geq Q^*\left( -\frac{\dot H(s)}{Q(df_{x(s)})}  \right) 
\end{equation}
since
\begin{eqnarray}
&& Q^*\left( -\frac{\dot H(s)}{Q(df_{x(s)})}  \right) =  
Q^*\left( \sum_{j=1}^m - p_jk_j^{\dagger}\frac{\tau_{x(s) \to x_0} \dot x(s)}{Q(df_{x(s)})}k_j\right) \leq 
\sum_{j=1}^m p_j Q^*\left( - k_j^{\dagger}\frac{\tau_{x(s) \to x_0} \dot x(s)}{Q(df_{x(s)})}k_j\right) \nonumber
\\
&& = \sum_{j=1}^m p_j Q^*\left( - \frac{\tau_{x(s) \to x_0} \dot x(s)}{Q(df_{x(s)})}\right)  
= Q^* \left( -\frac{\tau_{x(s) \to x_0} \dot x(s)}{Q(df_{x(s)})} \right) = Q^* \left( -\frac{\dot x(s)}{Q(df_{x(s)})} \right),  \label{eqn:as_in}
\end{eqnarray}
where the inequality follows from convexity and the next equality follows from unitary invariance.
Thus we have
\begin{eqnarray*}
&& \frac{1}{R(t)} 
\int_0^{t} Q^* \left( -\frac{\dot x(s)}{Q(df_{x(s)})} \right) Q(df_{x(s)}) ds
\geq \frac{1}{R(t)} 
\int_0^{t} Q^* \left( -\frac{\dot H(s)}{Q(df_{x(s)})} \right) Q(df_{x(s)}) ds \\
&& \geq Q^* \left( \frac{1}{R(t)}\int_{0}^t - \dot H(s) ds \right) = Q^*(- u(t)),
\end{eqnarray*}
where the second inequality is Jensen's inequality as before.
Then we have (\ref{eqn:becomes}), and complete the proof as before.

Case (iv). Suppose that ${\cal M} = {\cal M}_1 \times {\cal M}_2 \times \cdots \times {\cal M}_d$ 
where each  ${\cal M}_i$ satisfies (i), (ii), or (iii).
We write points $x \in {\cal M}$ as $x = (x^1,x^2,\ldots, x^d)$, where $x^i \in {\cal M}_i$. 
Also $T_x{\cal M} = T_{x^1}{\cal M}_1 \oplus \cdots \oplus T_{x^d}{\cal M}_d$.
The holonomy group $Hol_{x}({\cal M})$ is given by the product of 
$Hol_{x^i}({\cal M}_i)$. 
Represent $x(t)$ as $x(t) = \exp_{x_0} H(t) = (\exp_{x_0^1}H^1(t),\ldots, \exp_{x_0^d}H^d(t))$ as above. 
For each $i$ and $v^j \in T_{x_0^j}({\cal M}_j)$ $(j\neq i)$, we can apply the above arguments for $Hol_{x_0^i}({\cal M}_i)$-invariant $Q^*_{i,\{v^j\}}:H \mapsto Q^{*}(v^1,\ldots, v^{i-1}, H, v^{i+1}, \ldots,v^d)$ and obtain 
$Q^{*}_{i,\{v^j\}}(- \dot H^i(s)/Q(df_{x(s)})) 
\leq Q^*_{i,\{v^j\}}(- \dot x^i(s)/Q(df_{x(s)}))$.
From this, we have~(\ref{eqn:IEMI}) for ${\cal M}$. The rest is the same as above.
\end{proof}

We are going to prove Claim~\ref{clm:show}.
We first provide an explicit formula of $\dot H(s)$.
For $h \in \RR^n$, let $S^h = (S^h_{ij})$ be the $n \times n$ matrix defined by
\begin{equation*}
S^h_{ij} := \frac{(h_i-h_j)/2}{\sinh ((h_i-h_j)/2)} \quad (1\leq i,j \leq n).
\end{equation*}
For two matrices $A =(a_{ij})$, $B=(b_{ij})$ of the same size, 
the {\em Hadamard product (Schur product)} $A \odot B$ is defined by 
$(A \odot B)_{ij} := a_{ij}b_{ij}$.

\begin{Lem}[{See \cite[Part II, Lemma 10.40]{BridsonHaefliger1999}}]\label{lem:dotH(s)}
Suppose that  $x(s) = e^{H(s)}$, where 
$H(s) = k \diag h k^\dagger$ for $u \in SU_n$ and $h \in \RR^n$.
Then it holds
\begin{equation}\label{eqn:dotH(s)}
\dot H(s) = k (S^h \odot (k^\dagger \tau_{x(s) \to x_0} \dot x(s) k)) k^{\dagger}.
\end{equation}
\end{Lem}
\begin{proof}
Recall the parallel transport on ${\cal P}_n$ is given by 
\[
\tau_{x(s) \to x_0} \dot x(s) = x(s)^{-1/2} \dot x(s) x(s)^{-1/2}.
\]
Then, \cite[Part II, Lemma 10.40]{BridsonHaefliger1999} says that
\[
x(s)^{-1/2} \dot x(s) x(s)^{-1/2} = \frac{\sinh (\ad H(s)/2)}{\ad H(s)/2} \dot H(s),
\] 
where $(\ad H)(X) = [H,X] = HX - XH$ and $\frac{\sinh \ad H}{\ad H} (X):= \sum_{i=0}^\infty (-1)^k \frac{(\ad H)^{2k} (X)}{(2k+1)!}$.   
This formula can also be deduced by Jacobi field calculation 
with curvature tensor $R(X,Y)Z = - [X,Y]Z/4$ at $T_I {\cal P}_n$. 
From $(\ad H(s)) (\dot H(s)) = k ((\ad \diag h) (k^{\dagger}\dot H(s) k))k^\dagger$ 
and $(\ad \diag h) E_{ij} = (h_i- h_j)E_{ij}$, where $E_{ij}$ is the matrix unit, 
we obtain  
\[
(k^{\dagger}\tau_{x(s) \to x_0} \dot x(s)k)_{ij} 
= \frac{\sinh ((h_i-h_j)/2)}{(h_i-h_j)/2} (k^{\dagger} \dot H(s) k)_{ij},
\]
which implies (\ref{eqn:dotH(s)}).
\end{proof}
We combine the above formula (\ref{eqn:dotH(s)}) with methods 
of {\em matrix analysis}~\cite{Bhatia_PositiveDefiniteMatrices}.
Recall that a linear operator $\Psi:\CC^{n\times n} \to \CC^{n \times n}$ is said to be {\em doubly stochastic} if 
\begin{itemize}
    \item[(1)] $\Psi$ is {\em positive}, i.e., $\Psi(A) \in {\cal H}_n^+$ whenever $A \in {\cal H}_n^+$.
    \item[(2)] $\Psi$ is {\em unital}, i.e., $\Psi(I)  = I$.
    \item[(3)] $\Psi^*$ is also unital, i.e., $\Psi^*(I) = I$, 
\end{itemize}
where $\Psi^*$ is defined by $\trace \Psi(A)^{\dagger}B = \trace A^{\dagger} \Psi^*(B)$.

\begin{Lem}\label{lem:doublystochastic1}
For $k \in U_n$ and $h \in \RR^n$, 
the linear operator $\Psi$ defined by $A \mapsto k (S^h \odot (k^{\dagger} Ak)) k^{\dagger}$ 
is doubly stochastic.
\end{Lem}
\begin{proof}
(1). It is clear that $A \mapsto u A u^{\dagger}$ for any $u \in U_n$ is positive.
Since $x \mapsto x/\sinh x$ is a positive definite function~\cite[5.2.9]{Bhatia_PositiveDefiniteMatrices},  
it holds $S^h \in {\cal H}_n^+$. 
Therefore $A \mapsto S^h \odot A$
is positive~\cite[Example 2.2.1 (ix)]{Bhatia_PositiveDefiniteMatrices}.
Consequently, the composition $\Psi$ is also positive.

(2) follows from $S^h \odot I = I$.

(3). This follows from the observation  that $\Psi^*$ is given by $B \mapsto k^{\dagger} (S^h \odot (k Bk^{\dagger})) k$; see \cite[Exercise 2.7.11 (ii)]{Bhatia_PositiveDefiniteMatrices}.
\end{proof}

Now Claim~\ref{clm:show} is obtained from the next lemma in the setting of 
$A = \dot H(s)$, $B = \tau_{x(s) \to x_0} \dot x(s)$, $\Psi = k (S^h\odot (k^\dagger (\cdot)  k)) k^{\dagger}$.
\begin{Lem}[{\cite[Exercise 2.7.13]{Bhatia_PositiveDefiniteMatrices}}]\label{lem:doublystochastic2}
Let $A,B \in {\cal H}_n$. Suppose that $A = \Psi(B)$ for some doubly stochastic operator $\Psi$.
Then there are $k_1,k_2,\ldots,k_m \in U_n$ 
and $p_1,p_2,\ldots,p_m \in \RR_{>0}$ with $\sum_{j=1}^mp_j = 1$ such that
\[
A = \sum_{j=1}^m p_j k_j^{\dagger} B k_j.
\]
\end{Lem}
\begin{proof}
Suppose that $A = u \diag \lambda u^{\dagger}$ and $B = k \diag \mu k^{\dagger}$
for $u,k \in U_n$ and $\lambda,\mu \in \RR^n$. 
Define a linear map $\varphi:\RR^n \to \RR^n$ 
by $\varphi(z)_i:= (u^{\dagger} \Psi(k \diag z k^{\dagger}) u)_{ii}$.
Since $\Phi$ is positive, 
$\varphi$ maps $\RR^n_{\geq 0}$ to $\RR^n_{\geq 0}$.
Thus $\varphi$ is written as $\varphi(z) = M z$ for a nonnegative matrix $M$.
Since $\Psi$ and $\Psi^*$ are unital, $M$ is a doubly stochastic matrix.
By the Birkhoff-von Neumann theorem, $M$ is a convex combination $\sum_{j=1}^m p_j \sigma_j$ for permutation matrices $\sigma_1,\sigma_2,\ldots,\sigma_m$.
Thus $\diag \varphi(z) = \sum_{j=1}^m p_j \sigma_j (\diag z) \sigma_j^{\dagger}$.
By the assumption $A = \Psi(B)$, it holds $\varphi(\mu) = \lambda$, and
\[
A = u (\diag \varphi(\mu)) u^{\dagger} = 
\sum_{j=1}^m p_j u \sigma_j (\diag \mu) \sigma_j^{\dagger}u^{\dagger} = \sum_{j=1}^m p_j 
(u\sigma_j k^{\dagger}) B (k \sigma_j^{\dagger}  u^{\dagger}). \qedhere
\]
\end{proof}
 \begin{Conj}\label{conj}
 The conclusion of Theorem~\ref{thm:main} (and subsequent Theorem~\ref{thm:main2}) 
 holds for any Hadamard manifold.
 \end{Conj}
% \begin{Rem}
% %We conjecture that Theorem~\ref{thm:main} holds for any Hadamard manifolds.
% For proving this conjecture, 
% it suffices to show that 
% \begin{equation}\label{eqn:ggIEMI}
% F(\tau_{\exp_{x_0}H \to x_0} d(\exp_{x_0})_{H} K) \geq F(K) \quad (H,K \in T_{x_0})
% \end{equation}
% holds for any $Hol_{x_0}({\cal M})$-invariant convex function $F$ on $T_{x_0}$. 
% Indeed, it implies (\ref{eqn:IEMI}) in the setting of $F = Q^*_0$, 
% $H = H(s)$, and $K=-\dot H(s)/Q(df_{x(s)})$. 

% The property (\ref{eqn:ggIEMI}) may be interesting in its own right. 
% For the ${\cal P}_n^1$ case above, we have shown in (\ref{eqn:as_in}) that 
% \begin{equation*}
% F ( e^{-H/2} (d (e^{(\cdot)})_H K) e^{-H/2} ) \geq F(H) \quad (H,K \in {\cal H}_n).
% \end{equation*}
% holds for any unitarily invariant convex function 
% $F:{\cal H}_n \to \RR$.
% This property is viewed as a convex function version of
% the {\em generalized infinitesimal exponential metric
% increasing property (generalized IEMI)}~\cite[Proposition 6.4.1]{Bhatia_PositiveDefiniteMatrices}.
% \end{Rem}

Next, by using Theorem~\ref{thm:main}, 
we show that the strong duality itself holds 
for a more general class of $Q$. 
Relaxing Q1, consider
\begin{itemize}
\item[$\tilde {\rm Q}$1:]  For each $x \in {\cal M}$, 
the restriction $Q_x :T_x^* \to \overline{\RR}$ is 
a lower-semicontinuous convex function having bounded level sets.
%such that 
%$Q_x(p)= 0$ if and only if $p=0$.
\end{itemize}
Now suppose that $Q$ satisfies $\tilde {\rm Q}$1, Q2, and Q3,
and that $f$ satisfies 
\begin{equation*}
df_x \in \dom Q_x \quad (x \in {\cal M}). 
\end{equation*}

\begin{Thm}\label{thm:main2} Under this setting, 
assume that ${\cal M}$ is the form in Theorem~\ref{thm:main}. Then, it holds 
    \begin{equation}\label{eqn:strong_duality}
    \inf_{x \in {\cal M}} Q(df_x) = \sup_{\xi \in C{\cal M}^{\infty}} - f^{\infty}(\xi) - Q^*_{-}(\xi).  
    \end{equation}
\end{Thm}
\begin{proof}
We can assume that $Q$ is $\RR_{\geq 0}$-valued and $\inf_{x \in {\cal M}}Q(df_x)>0$  
(by adding a positive constant to $Q$).
%
%The weak duality (Lemma~\ref{lem:weakduality}) holds in this setting:
%\[
%\inf_{x \in {\cal M}} Q_x(df_x) \geq \sup_{\xi \in C{\cal M}^{\infty}} - f^{\infty}(\xi) - Q^*(\xi).
%\]
We utilize the {\em Moreau envelope} of $Q$; see e.g., \cite[Chapter 1.G]{RockafellarWets}.
Specifically, for $\lambda > 0$, define Moreau envelope $e_{\lambda} Q:T^*{\cal M} \to \RR$ by
\[
(e_\lambda Q)_x(p) := \inf_{q \in T_x^*} Q_x(q) + \frac{1}{2\lambda} (\|p-q\|^*)^2 \quad (p \in T_x^*, x\in {\cal M}).
\]
Then $(e_\lambda Q)_x$ is continuously differentiable convex (\cite[Theorem 2.26]{RockafellarWets}), and satisfies Q1.
The conjugate $(e_\lambda Q)^*$ is given by
\[
(e_\lambda Q)^*_x = Q_x^* + \frac{\lambda}{2} \|\cdot\|^2.
\]
See \cite[Example 11.26]{RockafellarWets}.
From this, we see that $e_\lambda Q (= (e_\lambda Q)^{**})$ satisfies Q2 and Q3.
By applying Theorem~\ref{thm:main} to $e_\lambda Q$, we have
\begin{equation}\label{eqn:strong_duality_e_lambdaQ}
\alpha_{\lambda} := \inf_{x\in {\cal M}} e_{\lambda}Q(df_{x}) = \sup_{\xi \in C{\cal M}^{\infty}} - f^{\infty}(\xi) - Q^*_{-}(\xi) - \frac{\lambda}{2} \|\xi\|^2 \quad (\lambda > 0). 
\end{equation}
From $Q \geq e_{\lambda}Q \geq e_{\lambda'}Q$ if $0 < \lambda \leq \lambda'$, 
we see that $\lambda \mapsto \alpha_{\lambda}$ is nonincreasing. 
Since it has an upper bound $\alpha^* :=  \inf_{x \in {\cal M}} Q(df_{x})$, 
the limit $\alpha_0 := \lim_{\lambda \to +0} \alpha_\lambda (= \sup_{\lambda >0} \alpha_\lambda)$ exists.
We show 
\[
\alpha_0 = \alpha^*
\]
Define $D_f \subseteq T_{x_0}^*$ by
\[
D_f := \{ \tau_{x \to x_0} df_{x} \mid x \in {\cal M} \}.
\]
Then 
\[
\alpha_{\lambda} = \inf_{p \in D_f} e_{\lambda}Q(p). 
\]
Since $e_{\lambda}Q$ has compact level sets, 
the infimum is attained by some point in $\overline{D_f}$ (closure of $D_f$). 
Since $\lv(e_{\lambda}Q,r) \subseteq \lv(e_{\lambda'}Q,r)$ for $\lambda \leq \lambda'$ and $r \geq 0$, 
we can choose convergent sequences $(\lambda_i)_{i=1,2,\ldots} \subseteq \RR_{>0}$ and
$(p_i)_{i=1,2,\ldots} \subseteq \overline{D_f}$ such that
$\lambda_i \to 0$, 
$e_{\lambda_i} Q(p_i) = \alpha_{\lambda_i}$, and $p_i \to p^*\in \overline{D_f}$.
%We show that $Q(p^*) = \alpha_0$.
%which implies $\alpha_0 = \alpha^*$ 
%(since $Q(p) = \alpha_0 \leq \alpha^* \leq Q(p)$)
Here $(\lambda,p) \mapsto e_{\lambda} Q(p)$ is lower semicontinuous, 
since it is written as
the supremum of affine functions  
$(p,\lambda) \mapsto  p(v) - Q^{*}(v)- (\lambda/2)\|v\|^2$ over $v \in T_x$.
Therefore 
$\alpha_0 = \lim_{i \to \infty}e_{\lambda_i} Q(p_i) \geq Q(p^*) \geq \inf_{p \in D_f} Q(p) = \alpha^* \geq \alpha_0$, 
which implies $\alpha_0 = \alpha^*$.
By the weak duality (Lemma~\ref{lem:weakduality}) and (\ref{eqn:strong_duality_e_lambdaQ}), we have 
\begin{eqnarray*}
&& \inf_{x \in {\cal M}} Q(df_x) \geq \sup_{\xi \in C{\cal M}^{\infty}} - f^{\infty}(\xi) - Q^*_{-}(\xi)  \geq  \sup_{\xi \in C{\cal M}^{\infty}} - f^{\infty}(\xi) - Q^*_{-}(\xi) - \frac{\lambda_i}{2} \|\xi\|^2  \\
&& = \inf_{x\in {\cal M}} e_{\lambda_i}Q(df_{x}) = \alpha_{\lambda_i} \quad \underset{i \to \infty}{\longrightarrow} \quad \alpha_0 = \alpha^* = \inf_{x \in {\cal M}} Q(df_x).
\end{eqnarray*}
Thus, the strong duality holds.
\end{proof}
We provide a geodesic convexity interpretation of the RHS in (\ref{eqn:strong_duality}).
\begin{Prop}\label{prop:g-convexity}
Assume that ${\cal M} = \RR^{n} \times {\cal P}_{n_1}^1 \times \cdots \times {\cal P}_{n_d}^1$ or $Q$ is written as 
$Q = h(\|\cdot\|^*)$.
Then, $f^{\infty} + Q_{-}^*$ is geodesically convex on $C{\cal M}^{\infty}$.
\end{Prop}
\begin{proof}
It is shown \cite{Hirai_Hadamard2022} 
that the recession function $f^{\infty}$ is geodesically convex on $C{\cal M}^{\infty}$.
Then the former case follows from (a straightforward generalization of) Lemma~\ref{lem:ui-convex=>g-convex}. 
Consider the latter case.  
In Hadamard space $C{\cal M}^{\infty}$,
the norm $\|\cdot\|$ represents the distance from $0$ 
and is geodesically convex~\cite[II.2.5]{BridsonHaefliger1999}. 
With the fact that $h^*$ is nondecreasing convex, 
we see geodesic convexity of $Q_{-}^* = h^{*}(\|\cdot\|)$.
\end{proof}

\paragraph{$Q$-subgradient method.}
Here we consider to solve (\ref{eqn:min_gradient-norm}) by discretization of the $Q$-gradient flow.
A natural discretization scheme of the $Q$-gradient flow is the following:
\begin{equation}\label{eqn:Q-subgradient_descent}
x_{i+1} := \exp_{x_i} - \delta_i Z_i,\quad Z_i \in  \partial (Q_{x_i}^2/2)(df_{x_i}) \quad (i=0,1,2,\ldots), 
\end{equation}
where $\delta_i >0$ is a step size.
By $\partial Q_{x_i}(df_{x_i}) = \tau_{I \to x_i}\partial Q_0(\tau_{x_i \to I}df_{x_i})$,  
the $Q$-subgradient method is implemented by subgradient computation of $Q_0$.

\begin{Question}
    Under which conditions of $f,Q,\delta_i$ does the sequence $Q(df_{x_i})$ converge to $\inf_{x\in {\cal M}}Q(df_{x})$ ? 
\end{Question}
Hirai and Sakabe~\cite{HiraiSakabe2024FOCS} showed that
the required convergence holds if $f$ is $L$-smooth convex, $Q = \|\cdot\|^*$, and $\delta_i=1/L$.
This important problem goes beyond the scope of this paper 
and has to be left to subsequent research.

\section{Convex optimization on entanglement polytopes}\label{sec:entanglement_polytope}
In this section, we apply the results in Section~\ref{sec:Q-gradient_flow}
to the Kempf-Ness optimization for GL-action on tensors.
% which is a geodesically convex function on the product of PD-manifolds.
% The corresponding minimum $Q$-gradient-norm problem (\ref{eqn:minimum_Q-gradient-norm_problem}) 
% is then convex optimization 
% on a class of moment polytopes, known as {\em entanglement polytopes}~\cite{WDGC2013_Science}.
For a positive integer $n$, let $[n] := \{1,2,\ldots,n\}$.
Let $Mat_n(\CC)$ denote the set of $n \times n$ matrices over $\CC$.
We will regard $Mat_{n_1}(\CC) \times Mat_{n_2}(\CC) \times \cdots \times Mat_{n_d}(\CC)$
as a block-diagonal subspace of $Mat_{n_1 + n_2 + \cdots + n_d}(\CC)$ by
\begin{equation}\label{eqn:block-diagonal}
(A_1,A_2,\ldots,A_d) \ \equiv  \
\left(
\begin{array}{cccc}
A_1 &     &  & \\
    & A_2 &  &  \\
    &     & \ddots & \\
    &   &   & A_d
\end{array}
\right),
\end{equation}
where multiplication and other operations are considered in $Mat_{n_1 + n_2 + \cdots + n_d}(\CC)$. 
%as $(A_1,A_2,\ldots,A_d) \cdot (B_1,B_2,\ldots,B_d) = (A_1B_1,A_2B_2,\ldots,A_dB_d)$. 

\subsection{GL-action on tensors and entanglement polytopes}
Let $d$ be a positive integer.
For each $i \in [d]$, let ${\cal V}_i := \CC^{n_i}$ be 
an $n_i$-dimensional $\CC$-vector space with standard Hermitian inner product 
$\langle u,v\rangle := u^{\dagger}v$.
Consider the tensor product space ${\cal V}:= {\cal V}_1 \otimes {\cal V}_2 \otimes \cdots \otimes {\cal V}_d$.
An element in ${\cal V}$ is called a ($d$-)tensor.
A tensor $v \in {\cal V}$ is uniquely written as 
\begin{equation*}%\label{eqn:v_j1j2...jd}
v = \sum_{j_1,j_2,\ldots,j_d} v_{j_1j_2\cdots j_d} e_{j_1} \otimes  e_{j_2} \otimes \cdots \otimes   e_{j_d},
\end{equation*}
where $e_i$ denotes the $i$-th unit vector and the index $j_k$ ranges over $[n_k]$.
The Hermitian inner product $\langle, \rangle$ on ${\cal V}$ is determined by extending 
$\langle u_1 \otimes \cdots \otimes u_d, v_1 \otimes \cdots \otimes v_d \rangle := \langle u_1,v_1\rangle  \cdots  \langle u_d,v_d\rangle$ (bi)linearly. The norm $\|\cdot\|$ on ${\cal V}$ 
is defined by $\|v\| := \sqrt{\langle v, v\rangle}$ accordingly.
The product of matrices $Mat_{n_1}(\CC) \times Mat_{n_2}(\CC) \times \cdots  \times  Mat_{n_d}(\CC)$
acts on ${\cal V}$ by
\[
(a_1,a_2,\ldots,a_d) \cdot (v_1 \otimes v_2 \otimes \cdots \otimes v_d) 
= (a_1 v_1) \otimes (a_2 v_2) \otimes \cdots \otimes (a_d v_d).  
\]
Let $G := GL_{n_1}(\CC) \times GL_{n_2}(\CC) \times \cdots \times GL_{n_d}(\CC)$. 
%and $K := SU_{n_1} \times SU_{n_2} \times \cdots \times SU_{n_d}$.
We consider the $G$-action on ${\cal V}$.
The associated {\em moment map} $\mu:{\cal V} \to  {\cal H} := {\cal H}_{n_1} \times {\cal H}_{n_2} \times \cdots \times {\cal H}_{n_d}$ is given by
\begin{equation*}
\mu(v) = \left( \mu_1(v), \mu_2(v),\ldots,\mu_d(v)\right), 
\end{equation*}
where each $\mu_k(v) \in {\cal H}_{n_k}$ is given by
\begin{equation*}
\trace \mu_k(v) H = \frac{\langle v, (I,\ldots, I, \overset{k}{H}, I,\ldots,I)\cdot v\rangle}{\|v \|^2} 
\quad (H \in {\cal H}_{n_k}).
\end{equation*}
% \begin{equation}
% \mu(v) = 
% \left(
% \begin{array}{cccc}
% \mu_1(v) & & & \\
%  & \mu_2(v) & & \\
%  &   & \ddots & \\
%  & & & \mu_d(v) 
% \end{array}
% \right),
%\end{equation}
%where each $\mu_i(v) \in V_i^{*} \times V_i$ is given by
%\begin{equation}
%\mu_i(v) = \frac{\trace^{-i}v^{\dagger} \otimes v}{\langle v, v \rangle}.
%\end{equation}
The matrix $\mu_k(v)$ is explicitly written as
\begin{equation}\label{eqn:mu_k(v)_ij}
\mu_k(v)_{ij}= \frac{1}{\| v \|^2}\sum_{j_1,\ldots, j_{k-1},j_{k+1}, \ldots,j_d}
v_{j_1 \cdots i \cdots j_d} \overline{v_{j_1 \cdots j \cdots j_d}} \quad (i,j \in [n_k]), 
\end{equation}
where the indices $i,j$ place on the $k$-th index and the index $j_\ell$ ranges over $[n_\ell]$ for $\ell \in [d]\setminus \{k\}$. 
 %and $\overline{(\cdot)}$ means the complex conjugate. 
% The matrix $\mu_k(v)$ is written by the {\em $k$-flatting} of $v$. 
% Specifically, regard $v$ as a linear map $v:({\cal V}_1 \otimes \cdots \otimes {\cal V}_{k-1} \otimes {\cal V}_{k+1} \otimes \cdots \otimes {\cal V}_d)^* \to {\cal V}_k$, 
% and consider its matrix representation by the basis of tensor product of unit vectors.
% For $i \in [n_k]$, let $v_i[k]$ be 
% the $\prod_{j \in [d]\setminus \{k\}} n_j$-dimensional vector
% consisting of $v_{j_1j_2\cdots j_{k-1}ij_{k+1}\cdots j_{d}}$, 
% where we fix an arbitrary ordering on indices $j_{1},j_2,\ldots,j_{k-1},j_{k+1},\ldots,j_{d}$.
% Then $\mu_k(v)$ is written as
% \begin{equation}
% \mu_k(v) = \left( \begin{array}{c} v_1[k]^{\dagger} \\ v_2[k]^{\dagger} \\ \vdots \\ v_{n_k}[k]^{\dagger}    \end{array}\right)
% \end{equation}
From this, we see that $\mu_i(v) \in {\cal D}^+_{n_i}$, 
and $\mu(v) \in {\cal D}^+:= {\cal D}^+_{n_1} \times {\cal D}^+_{n_2} \times \cdots \times {\cal D}^+_{n_d}$.
% \begin{eqnarray}
% \trace \mu_i(v) &=&  1, \\
% \mu_i(v) & \succeq & 0.
% \end{eqnarray}

Now we are given a convex function $S:{\cal H} \to \overline{\RR}$ 
that is invariant under the adjoint action of $K := U_{n_1} \times U_{n_2} \times \cdots \times U_{n_d}$:
\begin{itemize}
\item[S1:] $S$ is a lower-semicontinuous convex function with 
${\cal D}^+ \subseteq \dom S$.
\item[S2:] $S$ is {\em $K$-invariant}, that is, $S(Y) = S(k Y k^{\dagger})$  
holds for any $Y \in  {\cal H}$ and $k \in K$.
\end{itemize}
In S2, we use the convention (\ref{eqn:block-diagonal}).
Observe that $S$ is a function of the eigenvalues of $Y_i$.
For $(p^{(1)},p^{(2)},\ldots, p^{(d)}) \in \RR^{n_1} \times \RR^{n_2} \times \cdots \times \RR^{n_d}$, 
we denote
$S(\diag p^{(1)},\diag p^{(2)},\ldots, \diag p^{(d)})$ 
simply by $S(p^{(1)},p^{(2)},\ldots, p^{(d)})$. 

Consider the following optimization problem on the group $G$:
\begin{equation}\label{eqn:S(mu(gv))}
{\rm Min.} \quad S(\mu(g\cdot v)) \quad {\rm s.t.} \quad g\in G.
\end{equation}
This problem is viewed as a convex optimization problem on the Euclidean space $\RR^{n_1+\cdots + n_d}$ as well 
as the minimum $Q$-gradient-norm problem (\ref{eqn:minimum_Q-gradient-norm_problem}). 
For $i \in [d]$, 
let $\spec \mu_i(v) \in \RR^{n_i}$ denote the nonincreasing 
vector $p^{(i)} \in \RR^{n_i}$ with $p_1^{(i)} \geq p_2^{(i)} \geq \cdots \geq p_{n_i}^{(i)}$ 
of the eigenvalues of $\mu_i(v)$ counting multiplicity.
Let $\spec \mu(v) := (\spec \mu_1(v),\spec \mu_2(v),\ldots,\spec \mu_d(v))$.
Define $\Delta(v) \subseteq \RR^{n_1} \times \RR^{n_2} \times \cdots \times \RR^{n_d}$ by
\begin{equation}
\Delta(v) := \overline{\{ \spec \mu(g \cdot v)  \mid g \in G \}}. 
\end{equation}
Then the problem (\ref{eqn:S(mu(gv))}) is rewritten as 
\begin{equation}\label{eqn:min_S(p)}
{\rm Min.} \quad S(p) \quad {\rm s.t.} \quad p = (p^{(1)}, p^{(2)},\ldots,p^{(d)}) \in \Delta(v).
\end{equation}
The set $\Delta(v)$ is known as the {\em moment polytope} of the Hamiltonian action $G \curvearrowright \mathbb{P}({\cal V})$. 
Then, the classical result in GIT implies:
\begin{Thm}[{Convexity theorem~\cite{GuilleminSternberg1982-I,GuilleminSternberg1984-II,Kirwan1984}}]
$\Delta(v)$ is a convex polytope.
\end{Thm}
Although this theorem holds for more general actions, 
it was explored by Walter, Doran, Gross, and Christandl~\cite{WDGC2013_Science} that 
this moment polytope $\Delta(v)$ is an effective tool for studying the SLOCC entanglement class of a quantum multi-particle system,  
and is particularly called the {\em entanglement polytope}.
Note also that in the literature of K-stability, 
Lee, Sturm, and Wang~\cite{Lee2022momentmapconvexfunction} 
studied other aspects of invariant convex functions 
on moment polytopes (for general actions).

The problem (\ref{eqn:S(mu(gv))})  is also written as 
the minimum $Q$-gradient-norm problem (\ref{eqn:minimum_Q-gradient-norm_problem}) on 
the product of PD-manifolds ${\cal M} := {\cal P}_{n_1} \times {\cal P}_{n_2} \times \cdots \times {\cal P}_{n_d} 
\subseteq {\cal P}_{n_1 + \cdots + n_d}$.
Here $T_I =  T_I{\cal M}$ is given by ${\cal H}$, 
and $T_I^* = T_I^*{\cal M}$ is also identified with ${\cal H}$
by $Y(X) := \sum_{i=1}^d \trace Y_iX_i$ $(X \in T_I = {\cal H})$.
Now we can regard ${\cal D}^{+} \subseteq T_I^*$ and $S:T_I^* \to \overline{\RR}$.

Define the {\em Kempf-Ness function} $\Phi_v: {\cal M} \to \RR$ by
\begin{equation}
\Phi_v(x) := \log \langle v, x \cdot v\rangle \quad (x \in {\cal M}). 
\end{equation}
It is well-known that the Kempf-Ness function is a geodesically convex function on the symmetric space ${\cal M} \simeq G/K$.
\begin{Lem}[see e.g., \cite{BFGOWW_FOCS2018,BFGOWW_FOCS2019,GRS_Book}]
    $\Phi_v$ is geodesically convex.
\end{Lem}
The moment map $\mu$ is interpreted as transported differentials of $\Phi$.
\begin{Lem}[see \cite{BFGOWW_FOCS2018,HiraiSakabe2024FOCS}]\label{lem:dPhi_x}
$\tau_{x \to I}d(\Phi_v)_x = \mu( x^{1/2} \cdot v)$.
\end{Lem}
\begin{proof}
For $X = (X_1,X_2,\ldots,X_d) \in T_x$, we have
\begin{eqnarray*}
d(\Phi_v)_x(X) &= & \frac{d}{dt}\mid_{t = 0} \log \langle v , (x_1+tX_1,\ldots, x_d + tX_d) \cdot v \rangle \\ 
&=& \frac{1}{\| x^{1/2}\cdot v\|^2}\sum_{i=1}^d \langle x^{1/2} \cdot v , (I,\ldots,I,x_i^{-1/2} X_ix_i^{-1/2},I,\ldots,I) \cdot x^{1/2} \cdot v \rangle \\
&=& \sum_{i=1}^d \trace \mu_i(x^{1/2}\cdot v) (x_i^{-1/2}X_ix_i^{-1/2}).
\end{eqnarray*}
Hence $d(\Phi_v)_x(X) = \mu(x^{1/2}\cdot v)(\tau_{x \to I}X)$. 
%Projecting $\mu(v)$ to $T_I^*$, we obtain the claim. 
\end{proof}

Extend $S$ to $T^*{\cal M} \to \overline{\RR}$ by
\begin{equation*}
S_x(Y) := S\left(\tau_{x \to I} Y \right) \quad (x \in {\cal M}, Y \in T^*_x{\cal M}).
\end{equation*}
If $x = g^{\dagger}g$ for $g \in G$, 
then $x^{1/2} = k g$ for $k \in K$ and 
$S(d(\Phi_v)_x) = S(\tau_{x \to I} d (\Phi_v)_x) = S(\mu(x^{1/2}\cdot v)) = S(\mu(kg\cdot v)) = S(k^{\dagger} \mu(g\cdot v)k) = S(\mu(g \cdot v))$.
Hence, the problem (\ref{eqn:S(mu(gv))}) becomes
\begin{equation*}
    {\rm Min.} \quad S(d(\Phi_v)_x) \quad {\rm s.t.}\quad x\in {\cal M} = {\cal P}_{n_1} \times {\cal P}_{n_2} \times \cdots \times {\cal P}_{n_d}
\end{equation*}
for which the results in Section~\ref{sec:Q-gradient_flow} are applicable.
Since ${\cal D}^+$ is compact, we can assume that level sets of $S$ are bounded.
Then, $({\cal M},S)$ satisfies $\tilde {\rm Q}$1, Q2, and Q3 (see Example~\ref{ex:PD-manifold}).  
By Theorem~\ref{thm:main2}, we have:
\begin{Thm}[Duality theorem for convex optimization on entanglement polytope] \label{thm:duality_entanglement}
\begin{equation}\label{eqn:duality_entanglement}
\min_{p \in \Delta(v)} S(p) = \sup_{\xi \in C{\cal M}^{\infty}} -  \Phi_v^{\infty}(\xi) - S^*_{-}(\xi).
\end{equation}
\end{Thm}
We also give an explicit formula for the recession function $\Phi_v^{\infty}$.
Note that the recession function of the Kempf-Ness function is essentially Mumford's numerical invariant {\em $\mu$-weight $w_{\mu}$}~\cite[Definition 5.1]{GRS_Book}. 
\begin{Lem}[{Formula of $\Phi_v^{\infty}$; see \cite[Lemma 3.9 (1)]{Hirai_Hadamard2022}}]
   \begin{eqnarray*}
      &&  \Phi_v^{\infty}(Y)  = \max \{\lambda_{j_1}^{(1)}+ \lambda_{j_2}^{(2)} + \cdots + \lambda_{j_d}^{(d)} \mid j_1, j_2,\ldots, j_d: (k\cdot v)_{j_1j_2\cdots j_d}  \neq 0\},
   \end{eqnarray*}
   where  $k \in K$ such that
   $Y_i = k_i^{\dagger} \diag \lambda^{(i)} k_i$ and $\lambda^{(i)} \in \RR^{n_i}$ for $i\in [d]$.
\end{Lem}
\begin{proof}
By definition of $\Phi_v$ and $\Phi_v^{\infty}$, we have
\begin{eqnarray}
&& \Phi_v^{\infty}(Y) = \lim_{t \to \infty} \frac{\Phi_v(e^{tY})}{t} 
= \lim_{t \to \infty} \frac{1}{t}\log \| (e^{t \diag \lambda^{(1)}/2},\ldots,e^{t \diag \lambda^{(d)}/2}) k \cdot v \|^2 \nonumber \\
&& = \lim_{t \to \infty} \frac{1}{t} \log \sum_{j_1,j_2,\ldots,j_d} | ( k \cdot v)_{j_1j_2\cdots j_d}|^2 e^{t\left(\lambda^{(1)}_{j_1}+ \lambda^{(2)}_{j_2} + \cdots +\lambda^{(d)}_{j_d}\right)}. \label{eqn:Phi^infty}
\end{eqnarray}
Apply $\lim_{t \to \infty}(1/t)\log \sum_{i} e^{ta_i+b_i} = \max_{i} a_i$. 
\end{proof}
%This is true for 
%moment polytopes for general $SL$-actions. 
%If Conjecture~\ref{conj} is true for general symmetric spaces, 
%it is true for moment polytopes of reductive group actions.
%
By identifying $C{\cal M}^{\infty}$ with the Euclidean building 
${\cal B} := {\cal B}_{n_1} \times {\cal B}_{n_2} \times \cdots \times {\cal B}_{n_d}$,
the negative of the RHS of (\ref{eqn:duality_entanglement}) 
is formulated as a geodesically convex optimization on the building~${\cal B}$ (see also  Proposition~\ref{prop:g-convexity}):
\begin{eqnarray}
 \mbox{Max.} &&  -S^*(-(\lambda^{(1)}, \lambda^{(2)},\ldots, \lambda^{(d)})) \nonumber \\ 
 && \quad - \max \left\{\sum_{i=1}^d \lambda_{j_i}^{(i)} \mid j_1,j_2,\ldots,j_d: 
 v({\cal Y}^{(1)}_{j_1},{\cal Y}_{j_2}^{(2)},\ldots,{\cal Y}_{j_d}^{(d)}) \neq \{0\} \right\} \nonumber \\
 %&& \quad \quad \quad+  \sum_{i=1}^d \theta(i) \log \sum_{j=1}^{n_i} e^{- \lambda_j^{(i)}}, \nonumber\\
 \mbox{s.t.} && \{0\} < {\cal Y}_1^{(i)} <  {\cal Y}_2^{(i)} < \cdots < {\cal Y}_{n_i}^{(i)} = \CC^{n_i},\lambda_1^{(i)} \geq \lambda_2^{(i)} \geq \cdots \geq \lambda_{n_i}^{(i)} \ (i \in [d]),  \label{eqn:S(v)_building}
\end{eqnarray}
where $v$ 
is viewed as a multilinear map $\CC^{n_1} \times \CC^{n_2} \times \cdots \times \CC^{n_d} \to \CC$.

We expect that the duality (\ref{eqn:duality_entanglement}) would lead to 
a good characterization 
(NP $\cap$ co-NP characterization) for
the decision problem of deciding $\min_{p \in \Delta(v)} S(p) \leq \alpha$ or not.
Indeed, if $\min_{p \in \Delta(v)} S(p) > \alpha$, then 
a feasible solution $(\{{\cal Y}_j^{(i)}\}, \{\lambda_j^{(i)}\})$ of (\ref{eqn:S(v)_building}) with the objective value greater than $\alpha$
is a NO-certificate.
To accomplish this for specific problems, 
the bit-complexity of $(\{{\cal Y}_j^{(i)}\}, \{\lambda_j^{(i)}\})$ must be polynomially bounded, 
which is an important future research topic.
%By Proposition~\ref{prop:g-convexity}, the objective function is actually geodesically convex.

Next, consider the case where the $S$-gradient flow $x(t)$ of $\Phi_v$ is definable. 
By Theorem~\ref{thm:main}, 
$S(d(\Phi_v)_{x(t)}) = S(\mu(x^{1/2}(t)\cdot v))$ 
converges to the minimum $S$-gradient norm over ${\cal M}$, 
which implies that $\spec \mu(x^{1/2}(t) \cdot v)$ 
accumulates to minimizers of 
the problem (\ref{eqn:min_S(p)}).
\begin{Thm}[{$S$-moment-limit theorem}]\label{thm:moment-limit}
Suppose that $S$ is $\RR_{\geq 0}$-valued and $S^2$ is continuously differentiable.
Let $x(t)$ be the $S$-gradient flow of $\Phi_v$.
Then, any accumulation point of $\spec \mu(x^{1/2}(t) \cdot v)$ is a minimizer of $S$ over $\Delta (v)$.
%If a minimizer is unique, then $\spec \mu(x^{1/2}(t) \cdot v)$ converges to the unique minimizer.
\end{Thm}
From the view of the theory of {\em moment-weight inequality} (Georgoulas, Robbin, and Salamon\cite{GRS_Book}), 
this theorem is a generalization of the {\em moment-limit theorem}~\cite[Theorem 6.4]{GRS_Book}
from $\|\cdot\|_{\rm F}$ to $S$ (in our setting of the GL-action on tensors). 
We have verified that several parts of the theory
are directly generalized to the $S$-setting.
For example, the weak duality $\geq$ in (\ref{eqn:duality_entanglement}) 
is the $S$-version of the moment-weight inequality.
Also, the $S$-gradient flow is related to a generalization 
of the {\em gradient flow of moment-map squared}~\cite[Chapter 3]{GRS_Book} 
from $\|\cdot\|_{\rm F}^2/2$ to $S^2/2$.
However, we shall not pursue this issue further; the full generalization needs to solve Conjecture~\ref{conj}
for arbitrary symmetric spaces.

We expect that the $S$-moment-limit theorem will lead to an algorithm 
for convex optimization~(\ref{eqn:min_S(p)}). 
Indeed, it strongly suggests 
the convergence of the $S$-subgradient method with an appropriate step-size.
We here only give a formulation of the $S$-subgradient method 
by updating group element $g_i \in G$ in the spirit 
of {\em noncommutative ``group'' optimization}~\cite{BFGOWW_FOCS2019}: 
\begin{equation*}
g_{i+1} := e^{-\delta_i Z_i/2} g_i,\quad Z_i \in \partial (S^2/2) (\mu(g_i\cdot v)) \quad (i=1,2,\ldots).
\end{equation*}
Then $x_i := g_i^\dagger g_i$ satisfies (\ref{eqn:Q-subgradient_descent}).
This can be seen from Lemma~\ref{lem:dPhi_x} and 
$\partial (S^2/2) (\mu(k x_i^{1/2}\cdot v)) = \partial (S^2/2) (k\mu( x_i^{1/2}\cdot v)k^{\dagger}) 
= k^{\dagger} \partial (S^2/2) (\mu(x_i^{1/2}\cdot v)) k$
for  $g_i = k x_i^{1/2}$ with $k \in K$. 

In the subsequent sections, we consider three concrete examples 
(quantum functional, $G$-stable rank, noncommutative rank)
for which the developed theory is applicable.

\subsection{Quantum functional}
A $d$-tensor in $\bigotimes_{i=1}^d \CC^{n_i} = \CC^{n_1} \otimes \cdots \otimes \CC^{n_d}$ is viewed as a multilinear map $\CC^{n_1} \times \cdots \times \CC^{n_d} \to \CC$.
For two $d$-tensors $u \in \bigotimes_{i=1}^d \CC^{n_i}$ and 
$v \in \bigotimes_{i=1}^d \CC^{m_i}$, 
the tensor product $u \otimes v \in \bigotimes_{i=1}^d \CC^{n_i} \otimes \CC^{m_i}$ 
and the direct sum $u \oplus v \in 
 \bigotimes_{i=1}^d \CC^{n_i} \oplus \CC^{m_i}$ are defined 
by $(u \otimes v) (x_1\otimes y_1,\ldots,x_d \otimes y_d) := u(x_1,\ldots,x_d)v(y_1,\ldots,y_d)$ and $(u \oplus v)((x_1,y_1),\ldots,(x_d,y_d)) := u(x_1,\ldots,x_d)+ v(y_1,\ldots,y_d)$, respectively.
We say 
that $u$ {\em restricts to} $v$, denoted by $v \leq u$, if there are linear maps $A_k:\CC^{m_k} \to \CC^{n_k}$ 
such that $v = u \circ (A_1,A_2,\ldots,A_d)$.
A function $\xi: \{ \mbox{$d$-tensors} \} \to \RR_{\geq 0}$ is  
called a (universal) {\em spectral point}
if it is (a) monotone under $\leq$, (b) multiplicative under $\otimes$, (c) additive under $\oplus$,  
and (d) normalized so that $\xi(\langle n\rangle) = n$ for any unit tensor 
$\langle n \rangle := \sum_{i=1}^n e_i \otimes \cdots \otimes e_i \in (\CC^{n})^{\otimes d}$. 

Spectral points of tensors play significant roles in Strassen's theory~\cite{Strassen1988_asymptotic_spectrum}  
on {\em asymptotic spectra}
and matrix multiplication complexity; see \cite[part 3]{Landsberg_tensors2019} and \cite{WigdersonZuiddam2023} for survey.
Nontrivial spectral points (other than the flattening rank) were not known until
the recent discovery of {\em quantum functionals} by 
Christandl, Vrana, and Zuiddam~\cite{CristandlVramaZuiddam2023}.
For a probability vector $\theta \in \RR_{>0}^d$ with $\sum_{i=1}^d \theta_i = 1$,  
the quantum functional $F_{\theta}: \{ \mbox{$d$-tensors} \} \to \RR_{\geq 0}$ is defined by
\begin{equation}\label{eqn:quantum_functional}
F_{\theta}(v) := 2^{E_{\theta}(v)},\quad E_{\theta}(v) := \sup_{g \in G} \sum_{i=1}^d \theta_i H (\mu_i(g \cdot v)), 
\end{equation}
where $H$ is the von Neumann entropy (Example~\ref{ex:entropy}). 
It is clear that
the logarithmic quantum functional $E_{\theta}$ is obtained by 
entropy maximization on the entanglement polytope: 
\begin{equation*}
{\rm Max.} \quad  \sum_{i=1}^d \theta_i H (p^{(i)}) \quad {\rm s.t.} \quad p=(p^{(1)},p^{(2)},\ldots,p^{(d)})\in \Delta(v).
\end{equation*}
By applying Theorem~\ref{thm:duality_entanglement} 
for the setting $S(Y) := - \sum_{i=1}^d \theta_i H(Y_i)$, we obtain 
a new expression of the logarithmic quantum functional $E_{\theta}$.
\begin{Thm}
\begin{equation}\label{eqn:E_theta(v)}
E_\theta(v) = \inf_{X \in {\cal H}} \Phi_v^{\infty}(X) + \sum_{i=1}^d \theta_i \log_2 \trace 2^{- X_i/\theta_i}.
\end{equation}
\end{Thm}
In this formula, we identify $C{\cal M}^{\infty}$ with $T_I = {\cal H}$.
See Example~\ref{ex:entropy} for the conjugate of $-H$.

\paragraph{Seeing that quantum functionals are spectral points.}
It may be worth seeing the properties (a),(b),(c),(d) of quantum functionals from our approach.
As in \cite[Section 3.1]{CristandlVramaZuiddam2023}, 
the properties (a),(d), super-multiplicativity $F_{\theta}(v \otimes v') \geq F_{\theta}(v) F_{\theta}(v')$ in (b), and 
super-additivity $F_{\theta}(v \oplus v') \geq F_{\theta}(v) + F_{\theta}(v')$ in (c) 
are directly shown from the definition (\ref{eqn:quantum_functional}).
The difficult part (for us) is to show the remaining properties $F_{\theta}(v \otimes v') \leq F_{\theta}(v) F_{\theta}(v')$
and $F_{\theta}(v \oplus v') \leq F_{\theta}(v) + F_{\theta}(v')$, which requires 
representation-theoretic methods~\cite[Section 3.2]{CristandlVramaZuiddam2023} \cite[Sections 7.3.1, 7.3.2]{Landsberg_tensors2019}. 

The above formula (\ref{eqn:E_theta(v)}) directly proves sub-multiplicativity of $F_{\theta}$ (sub-additivity of $E_{\theta}$) as follows.
Let $X = (X_i)$ and $X' = (X_i')$ be any solutions of the optimization problems in (\ref{eqn:E_theta(v)}) for
$E_\theta(v)$ and $E_\theta(v')$, respectively. 
Consider $Z = (Z_i) := (X_i \otimes I + I \otimes X_i')$, 
which is a solution of the problem for $E_\theta(v \otimes v')$.
Then we observe $\Phi_v^{\infty}(Z) = \Phi_v^{\infty}(X) + \Phi_v^{\infty}(X')$ 
and  $\log_2 \trace 2^{-Z_i/\theta_i} = \log_2 \trace 2^{-X_i/\theta_i} + \log_2 \trace 2^{-X_i'/\theta_i}$.
From this, we have sub-additivity
$E_{\theta}(v \otimes v') \leq \inf_{Z = (X_i \otimes I + I \otimes X_i')} \Phi_v^{\infty}(Z) + 
\sum_{i} \theta_i \log_2 \trace 2^{Z_i/\theta_i} = E_{\theta}(v)+ E_{\theta}(v')$.

The additivity (c) can be seen from the convergence of the $S$-gradient flow 
(Theorem~\ref{thm:moment-limit}), where we can choose $S$ 
as the Moreau envelop 
of $- E_{\theta} + C > 0$ 
for a constant $C$ and an arbitrary small parameter $\lambda >0$; see the proof of Theorem~\ref{thm:main2}.  
Let ${\cal V}'$, $G'$, and ${\cal M}'$ denote the tensor space, group, and manifold, respectively, for the second tensor $v'$.
Let $\bar{G}$ and $\bar{\cal M}$ denote the group and manifold for the sum $v \oplus v'$, respectively. 
By direct calculation $2^{\sup_{g \in G,g'\in G'} \sum_{i} \theta(i) H(\mu_i( (g\cdot v) \oplus (g' \cdot v')))} 
= F_\theta(v) + F_\theta(v')$~\cite[Lemmas 3.8--3.10]{CristandlVramaZuiddam2023},
it suffices to show 
\begin{equation}\label{eqn:sup_h=sup_g,g'}
(E_{\theta}(v \oplus v') :=)\  \sup_{\bar g \in \bar G} \sum_{i=1}^d \theta_i H(\mu_i( \bar g \cdot (v \oplus v')))  = \sup_{g \in G,g'\in G'} \sum_{i=1}^d \theta_i H(\mu_i( (g\cdot v) \oplus (g' \cdot v'))).
\end{equation}
Here, the moment map $\mu( (g\cdot v) \oplus (g'\cdot v')) \in T_{I}^*\bar{\cal M}$ is written as
\[
\mu((g\cdot v) \oplus (g'\cdot v')) = \left(  
\begin{array}{cc}
\frac{\|g\cdot v\|^2}{\|g\cdot v\|^2+\|g'\cdot v'\|^2}\mu(g\cdot v) & O \\
O & \frac{\|g'\cdot v'\|^2}{\|g\cdot v\|^2+\|g'\cdot v'\|^2}\mu(g'\cdot v')
\end{array}
\right) \quad \in T_{I,I}^*({\cal M} \times {\cal M}').
\]
From this matrix structure and the formula of $dF$ of a unitarily invariant convex function $F$ (Section~\ref{subsec:convex-analysis}), 
we see that 
the $S$-gradient $\nabla^{S} \Phi_{v \oplus v'}(x,x') = \tau_{(I,I) \to (x,x')} d (S^2/2)_{\mu((x^{1/2}\cdot v) \oplus ((x')^{1/2}\cdot v')))}$ at the submanifold ${{\cal M} \times {\cal M}'} \subseteq \bar{\cal M}$ belong to $T({\cal M} \times {\cal M}')$.
Consequently, the $S$-gradient flow of $\Phi_{v \oplus v'}$
with initial point $x_0 = I \in {\cal M} \times {\cal M}'$ 
remains in ${\cal M} \times {\cal M}'$ but attains the minimum $S$-gradient-norm over $\bar{\cal M}$ in the limit.
This implies (\ref{eqn:sup_h=sup_g,g'}) by $\sup_{\bar g \in \bar{G}} = \sup_{\bar x \in \bar {\cal M}} 
= \sup_{(x,x') \in {\cal M} \times {\cal M}'} = \sup_{g\in G, g' \in G'}$.

%van den Berg et al.\cite{vdBCLNWZ_STOC2025} 

\subsection{$G$-stable rank}
Derksen~\cite{Derksen2022ANT} introduced a new rank concept of tensors, 
called the {\em $G$-stable rank}, 
and revealed interesting properties and applications. 
As a consequence of~\cite[Theorem 5.2]{Derksen2022ANT}, 
the $G$-stable rank for a complex $d$-tensor $v \in {\cal V}$ 
is obtained by a convex optimization problem on the entanglement polytope.
Here we introduce the $G$-stable rank via this optimization problem.
For each $i\in [d]$, 
the tensor $v$ is viewed as a linear map 
$v:({\cal V}_1 \otimes \cdots \otimes {\cal V}_{i-1} \otimes {\cal V}_{i+1} \otimes \cdots \otimes {\cal V}_d)^* \to {\cal V}_i$, which is called the {\em $i$-flattening} of $v$. 
The corresponding operator norm is denoted by  $\|v\|_{i, {\rm op}}$.
For a positive vector $\alpha \in \RR^d_{>0}$, 
the $G$-stable rank $\rk^G_\alpha (v)$ of $v$ is defined by
\begin{equation*}
\rk^G_{\alpha}(v) := \sup_{g \in G} \min_{i \in [d]} \frac{\alpha_i \| g\cdot v \|^2}{\|g \cdot v\|^2_{i,{\rm op}}}.
\end{equation*}
Consider the matrix representation $A_{v,i} = (v_{j_1\cdots j_i \cdots j_d})$ 
of the $i$-flattening of $v$, where 
$v_{j_1\cdots j_i \cdots j_d}$ is the element of $A_{v,i}$ with the row index $j_i$ 
and column index $(j_1,\ldots, j_{i-1},j_{i+1},\ldots, j_d)$.
Then we see from (\ref{eqn:mu_k(v)_ij}) that $\mu_i(v) = A_{v,i} A^{\dagger}_{v,i}/ \langle v,v \rangle$, and 
\[
\frac{\|v\|^2_{i,\rm{op}}}{\|v\|^2} = \|\mu_i(v)\|_{\rm op}.
\]
% Therefore  
% $\rk^G (v)$  is written as
% \begin{equation}
% \rk^G(v) = \sup_{g \in G} \min_{i \in [d]} \frac{\alpha_i}{\|\mu_i(g \cdot v)\|_{\rm op}}
% \end{equation}
Therefore, $1/ \rk_{\alpha}^G(v)$ is given by convex optimization~(\ref{eqn:S(mu(gv))}) 
with the setting $S(Y) := \max_{i \in [d]}\|Y_i\|_{\rm op}/\alpha_i$:
\begin{equation*}
{\rm Min.} \quad \max_{i\in [d]} \frac{1}{\alpha_i} \|\mu_i(g \cdot v)\|_{\rm op}  \quad {\rm s.t.} \quad g\in G.
\end{equation*}
It is a weighted $\ell_{\infty}$-norm minimization problem 
on the entanglement polytope: 
\begin{equation*}
{\rm Min.} \quad \max_{i\in [d]} \frac{1}{\alpha_i} \|p^{(i)}\|_{\rm \infty}  \quad {\rm s.t.} \quad p=(p^{(1)},p^{(2)},\ldots,p^{(d)})\in \Delta(v).
\end{equation*}
By applying Theorem~\ref{thm:duality_entanglement},
we obtain a dual  expression of $\rk^G_{\alpha}$: 
\begin{Thm}%[{see \cite[Theorem 5.2]{Derksen2022ANT}}]
\begin{equation}\label{eqn:1/rk^G}
\frac{1}{\rk^G_{\alpha}(v)} = \sup \left\{ - \Phi_v^{\infty}(X) \mid X \in {\cal H}:\sum_{i=1}^d \alpha_i \| X_i \|_{\rm tr} \leq 1  \right\}.
\end{equation}
\end{Thm}
\begin{proof}
Since $S$ is a norm, the conjugate $S^*$ is given by $\delta_B$ for
the unit ball $B$ with respect to the dual norm of $S$, which 
is given by $X \mapsto 
\sup_{Y_i: \|Y_i/\alpha_i\|_{\rm op}\leq 1} \sum_{i=1}^d \trace X_i Y_i = \sum_{i=1}^d  \alpha_i \|X_i\|_{\rm tr}$ (see Example~\ref{ex:norm}).
%Then the LHS of (\ref{eqn:duality_entanglement}) 
%is $\sup_{X \in {\cal H}}- \Phi^{\infty}(X) - \delta_{B}$.
\end{proof}
This formula matches the original definition of 
$G$-stable rank~\cite[Definition 1.3, Theorem 2.4]{Derksen2022ANT}.\footnote{
By using notation $\val_{s}$ in $\cite{Derksen2022ANT}$,  
the correspondence is given by $\Phi^{\infty}(Y) = - \val_{s} (s^{-Y} \cdot v)$ (seen from (\ref{eqn:Phi^infty}))  
and $\|Y_k\|_{\rm tr} = - \val_{s} \det s^{-Y_i}$, 
where the eigenvalues $\lambda^{(i)}_j$ of $Y_i$ in (\ref{eqn:1/rk^G}) can be assumed to be nonpositive}

\subsection{Noncommutative rank}\label{subsec:ncrank}
Finally, we explain the noncommutative rank from our viewpoint. 
Consider a linear symbolic matrix 
\[
A = \sum_{k=1}^m A_k x_k,
\]
where $A_k \in Mat_n(\CC)$ and $x_k$ is an indeterminate for $k \in [m]$.
The {\em noncommutative rank} $\ncrk A$~(Ivanyos, Qiao, and Subrahmanyam~\cite{IQS2017}) 
is the rank of $A$ considered in the {\em free skew field}~\cite{Cohn} 
generated by $x_1,x_2,\ldots,x_m$ being pairwise noncommutative $x_ix_j \neq x_jx_i$. 
%of the noncommutative polynomial ring $\CC\langle x_1,x_2,\ldots,x_m\rangle$.
%under the setting where $x_i$ are noncommutative $x_{i}x_j \neq x_jx_i$. 
%which is defined in the {\em free skew field}---the most generic skew field of 
%fractions of the noncommutative polynomial ring $\CC\langle x_1,x_2,\ldots,x_m\rangle$~\cite{IQS2017}.
%
The polynomial-time computation of $\ncrk A$ by \cite{GGOW,HamadaHirai2021,IQS2018} 
is the starting point of what we have developed here.
It is shown by Fortin and Reutenauer~\cite{FortinReutenauer04} that $\ncrk A$ is given by a vector-space optimization problem:
\begin{Thm}[\cite{FortinReutenauer04}]\label{thm:FR}
\begin{equation}\label{eqn:ncrank}
\ncrk A = 2n  - \max \{ \dim {\cal X} + \dim {\cal Y} \mid {\cal X},{\cal Y} \leq \CC^n: A_k({\cal X},{\cal Y}) = \{0\}\ (k \in [m])\}, 
\end{equation}
where each $A_k$ is viewed as a bilinear form $\CC^n \times \CC^n \to \CC$ by $(u,v) \mapsto u^{\top}A_kv$.
\end{Thm}
We regard the matrix $A$ as a 3-tensor $A \in \CC^n \otimes \CC^n \otimes \CC^m$ by
\begin{equation*}
A = \sum_{i,j \in [n],k \in [m]} (A_k)_{ij} e_i \otimes e_j \otimes e_k.
\end{equation*}
Consider the $G$-action for $G = GL_n(\CC) \times GL_n(\CC) \times GL_m(\CC)$ on $A$, as above.
Although it is shown \cite[Proposition 2.9]{Derksen2022ANT} that 
the noncommutative rank
$\ncrk A$ coincides with the $G$-stable rank $\rk_{\alpha}^G A$ for $\alpha = (1,1,n)$, 
we explain a different approach \cite{Hirai2025}  
relating $\ncrk A$ with optimization on a moment polytope.
We assume for simplicity that the common left or right kernel of $A_k$ is trivial:
\begin{equation}\label{eqn:commonkernel}
\bigcap_{k=1}^m \ker A_k = \{0\} \quad \mbox{or} \quad \bigcap_{k=1}^m \ker A_k^{\dagger} = \{0\}.
\end{equation}
Otherwise, $\ncrk A$ reduces to smaller matrices. 
We restrict the $G$-action 
to $GL_n(\CC) \times GL_n(\CC) \times \{I\}$, which is equal to the {\em left-right action}.
The first two components of the moment map $\mu(A) = (\mu_1(A),\mu_2(A),\mu_3(A))$ 
are given by
\begin{equation*}
 \mu_1(A) = \frac{1}{\|A\|^2} \sum_{k=1}^m A_k A_k^{\dagger},\quad \mu_2(A)^{\top} = \frac{1}{\| A\|^2} \sum_{k=1}^m A_k^{\dagger} A_k.
\end{equation*}
Then $\ncrk A$ is written as:
\begin{Thm}[\cite{Hirai2025}]
\begin{equation}\label{eqn:ncrank2}
\ncrk A = n - \frac{n}{2} \inf_{g_1,g_2 \in GL_n(\CC)} \left\| \mu_1((g_1,g_2,I) \cdot A) - \frac{1}{n}I \right\|_{\rm tr} + \left\| \mu_2((g_1,g_2,I) \cdot A) - \frac{1}{n}I \right\|_{\rm tr}. 
\end{equation}
% \begin{eqnarray*}
% n - ncrk A &=& \frac{n}{2} \inf_{g_1,g_2 \in GL_n(\CC)} \left\| \mu_1((g_1,g_2,I) \cdot A) - \frac{1}{n}I \right\|_{\rm tr} + \left\| \mu_2((g_1,g_2,I) \cdot A) - \frac{1}{n}I \right\|_{\rm tr} \\
% &=& \frac{n}{2} \inf_{(p_1,p_2) \in \Delta(A)} \|p_1 -  {\bf 1}/n \|_1 + \|p_2 - {\bf 1}/n \|_1. 
% \end{eqnarray*}
\end{Thm}
Although the formula in~\cite{Hirai2025} 
does not have the denominator $\|(g_1,g_2,I) \cdot A\|^2$ of $\mu_i((g_1,g_2,I) \cdot A)$, 
by the assumption~(\ref{eqn:commonkernel}), 
we may restrict $g_1,g_2$ to satisfy $\|(g_1,g_2,I) \cdot A\|^2 = n$ in the infimum; 
see \cite[(4.8)]{Hirai2025}.

Consider the following variant of the entanglement polytope 
\[
\Delta'(A) := \overline{\{ (\mu_1( (g_1,g_2,I)\cdot A),\mu_2((g_1,g_2,I)\cdot A)    \mid g_1,g_2\in GL_n(\CC)\}},
\]
which is also a convex polytope; it is the moment polytope for the left-right action.
Therefore, $\ncrk A$ is obtained from $\ell_1$-norm minimization on $\Delta'(A)$:
\begin{equation*}
\mbox{Min.} \quad  \left\|p^{(1)} -  \frac{\bf 1}{n} \right\|_1 
+ \left\|p^{(2)} - \frac{\bf 1}{n} \right\|_1  \quad \mbox{s.t.} \quad (p^{(1)},p^{(2)}) 
\in \Delta'(A).
\end{equation*}
In particular, the relation of the above two formulas (\ref{eqn:ncrank}) and (\ref{eqn:ncrank2}) of nc-rank can be understood from the duality in Theorem~\ref{thm:duality_entanglement}. 
The nc-rank computation algorithm by
Hamada and Hirai~\cite{HamadaHirai2021} solves (a variant of)
the problem (\ref{eqn:S(v)_building}), 
which is a piecewise linear extension ({\em Lov\'asz extension}~\cite{Hirai_Hadamard2022}) of 
the optimization problem in (\ref{eqn:ncrank}); see also~\cite{Hirai2025}.

\section*{Acknowledgments}
The author thanks Keiya Sakabe, Asuka Takatsu, Shinichi Ohta, Yoshinori Hashimoto for discussion,  
Cheng Zhuohao for comments, and Michael Walter for suggesting paper~\cite{Lee2022momentmapconvexfunction}.
The author was supported by JSPS KAKENHI Grant Number JP24K21315.
\bibliographystyle{plain}
%\small
\bibliography{G-gradientflow}

\begin{thebibliography}{10}

\bibitem{Ambrosio_Book}
N.~Ambrosio, L.~Gigli and G.~Savar\'{e}.
\newblock {\em Gradient Flows: In Metric Spaces and in the Space of Probability
  Measures}.
\newblock Birkh\"{a}user, Basel, {S}econd edition, 2008.

\bibitem{Bhatia_PositiveDefiniteMatrices}
R.~Bhatia.
\newblock {\em Positive Definite Matrices}.
\newblock Princeton University Press, Princeton, NJ, 2007.

\bibitem{Boumal_Book}
N.~Boumal.
\newblock {\em An Introduction to Optimization on Smooth Manifolds}.
\newblock Cambridge University Press, Cambridge, 2023.

\bibitem{BridsonHaefliger1999}
M.~R. Bridson and A.~Haefliger.
\newblock {\em Metric Spaces of Non-Positive Curvature}.
\newblock Springer-Verlag, Berlin, 1999.

\bibitem{Burgisser_ACT_Book}
P.~B\"urgisser, M.~Clausen, and M.~A. Shokrollahi.
\newblock {\em Algebraic Complexity Theory}.
\newblock Springer-Verlag, Berlin, 1997.

\bibitem{BFGOWW_FOCS2018}
P.~B{\"{u}}rgisser, C.~Franks, A.~Garg, R.~Oliveira, M.~Walter, and
  A.~Wigderson.
\newblock Efficient algorithms for tensor scaling, quantum marginals, and
  moment polytopes.
\newblock In {\em 59th {IEEE} Annual Symposium on Foundations of Computer
  Science, {FOCS} 2018}, pages 883--897, 2018.

\bibitem{BFGOWW_FOCS2019}
P.~B\"{u}rgisser, C.~Franks, A.~Garg, R.~Oliveira, M.~Walter, and A.~Wigderson.
\newblock Towards a theory of non-commutative optimization: geodesic 1st and
  2nd order methods for moment maps and polytopes.
\newblock In {\em 60th {IEEE} {A}nnual {S}ymposium on {F}oundations of
  {C}omputer {S}cience, FOCS 2019}, pages 845--861, 2019.

\bibitem{CristandlVramaZuiddam2023}
M.~Christandl, P.~Vrana, and J.~Zuiddam.
\newblock Universal points in the asymptotic spectrum of tensors.
\newblock {\em J. Amer. Math. Soc.}, 36(1):31--79, 2023.

\bibitem{Cohn}
P.~M. Cohn.
\newblock {\em Skew Fields: Theory of General Division Rings}.
\newblock Cambridge University Press, Cambridge, 1995.

\bibitem{Davis1957}
C.~Davis.
\newblock All convex invariant functions of hermitian matrices.
\newblock {\em Arch. Math.}, 8:276--278, 1957.

\bibitem{Derksen2022ANT}
H.~Derksen.
\newblock The {$G$}-stable rank for tensors and the cap set problem.
\newblock {\em Algebra Number Theory}, 16(5):1071--1097, 2022.

\bibitem{Eberlein}
P.~B. Eberlein.
\newblock {\em Geometry of Nonpositively Curved Manifolds}.
\newblock University of Chicago Press, Chicago, IL, 1996.

\bibitem{FortinReutenauer04}
M.~Fortin and C.~Reutenauer.
\newblock Commutative/non-commutative rank of linear matrices and subspaces of
  matrices of low rank.
\newblock {\em S\'em. Lothar. Combin.}, 52:B52f, 2004.

\bibitem{GGOW}
A.~Garg, L.~Gurvits, R.~Oliveira, and A.~Wigderson.
\newblock Operator scaling: theory and applications.
\newblock {\em Found. Comput. Math.}, 20(2):223--290, 2020.

\bibitem{GRS_Book}
V.~Georgoulas, J.~W. Robbin, and D.~A. Salamon.
\newblock {\em The Moment-Weight Inequality and the {H}ilbert-{M}umford
  Criterion---{GIT} from the Differential Geometric Viewpoint}, volume 2297 of
  {\em Lecture Notes in Mathematics}.
\newblock Springer, Cham, 2021.

\bibitem{GuilleminSternberg1982-I}
V.~Guillemin and S.~Sternberg.
\newblock Convexity properties of the moment mapping.
\newblock {\em Invent. Math.}, 67(3):491--513, 1982.

\bibitem{GuilleminSternberg1984-II}
V.~Guillemin and S.~Sternberg.
\newblock Convexity properties of the moment mapping. {II}.
\newblock {\em Invent. Math.}, 77(3):533--546, 1984.

\bibitem{Gurvits2004}
L.~Gurvits.
\newblock Classical complexity and quantum entanglement.
\newblock {\em J. Comput. System Sci.}, 69(3):448--484, 2004.

\bibitem{HamadaHirai2021}
M.~Hamada and H.~Hirai.
\newblock Computing the nc-rank via discrete convex optimization on {${\rm
  CAT}(0)$} spaces.
\newblock {\em SIAM J. Appl. Algebra Geom.}, 5(3):455--478, 2021.

\bibitem{Hirai_Hadamard2022}
H.~Hirai.
\newblock Convex analysis on {H}adamard spaces and scaling problems.
\newblock {\em Found. Comput. Math.}, 24:1979--2016, 2024.

\bibitem{Hirai2025}
H.~Hirai.
\newblock A scaling characterization of nc-rank via unbounded gradient flow.
\newblock {\em Linear Algebra Appl.}, 730:525--545, 2026.

\bibitem{HiraiSakabe2024FOCS}
H.~Hirai and K.~Sakabe.
\newblock Gradient descent for unbounded convex functions on {H}adamard
  manifolds and its applications to scaling problems.
\newblock In {\em 65th {IEEE} Annual Symposium on Foundations of Computer
  Science, {FOCS} 2024}, pages 2387--2402. {IEEE}, 2024.
\newblock {\tt arXiv:2404.09746}.

\bibitem{IQS2017}
G.~Ivanyos, Y.~Qiao, and K.~V. Subrahmanyam.
\newblock Non-commutative {E}dmonds' problem and matrix semi-invariants.
\newblock {\em Comput. Complex.}, 26(3):717--763, 2017.

\bibitem{IQS2018}
G.~Ivanyos, Y.~Qiao, and K.~V. Subrahmanyam.
\newblock Constructive non-commutative rank computation is in deterministic
  polynomial time.
\newblock {\em Comput. Complex.}, 27(4):561--593, 2018.

\bibitem{KLM2009JDG}
M.~Kapovich, B.~Leeb, and J.~Millson.
\newblock Convex functions on symmetric spaces, side lengths of polygons and
  the stability inequalities for weighted configurations at infinity.
\newblock {\em J. Differential Geom.}, 81(2):297--354, 2009.

\bibitem{Kirwan1984}
F.~Kirwan.
\newblock Convexity properties of the moment mapping. {III}.
\newblock {\em Invent. Math.}, 77(3):547--552, 1984.

\bibitem{Landsberg_tensors}
J.~M. Landsberg.
\newblock {\em Tensors: {G}eometry and {A}pplications}.
\newblock American Mathematical Society, Providence, RI, 2012.

\bibitem{Landsberg_tensors2019}
J.~M. Landsberg.
\newblock {\em Tensors: {A}symptotic {G}eometry and {D}evelopments 2016--2018}.
\newblock American Mathematical Society, Providence, RI, 2019.

\bibitem{Lee2022momentmapconvexfunction}
K.~L. Lee, J.~Sturm, and X.~Wang.
\newblock Moment map, convex function and extremal point, 2022.
\newblock arXiv:2208.03724.

\bibitem{Lewis1996}
A.~S. Lewis.
\newblock Convex analysis on the {H}ermitian matrices.
\newblock {\em SIAM J. Optim.}, 6:164--177, 1996.

\bibitem{Petersen}
P.~Petersen.
\newblock {\em Riemannian {G}eometry}.
\newblock Springer, Cham, third edition, 2016.

\bibitem{Rockafellar}
R.~T. Rockafellar.
\newblock {\em Convex Analysis}.
\newblock Princeton University Press, Princeton, NJ, 1970.

\bibitem{RockafellarWets}
R.~T. Rockafellar and R.~J-B. Wets.
\newblock {\em Variational Analysis}.
\newblock Springer-Verlag Berlin, Heidelberg, 1998.

\bibitem{Sakai1996}
T.~Sakai.
\newblock {\em Riemannian Geometry}.
\newblock American Mathematical Society, Providence, RI, 1996.

\bibitem{Strassen1988_asymptotic_spectrum}
V.~Strassen.
\newblock The asymptotic spectrum of tensors.
\newblock {\em J. Reine Angew. Math.}, 384:102--152, 1988.

\bibitem{vdBCLNWZ_STOC2025}
M.~van~den Berg, M.~Christandl, V.~Lysikov, H.~Nieuwboer, M.~Walter, and
  J.~Zuiddam.
\newblock Computing moment polytopes of tensors, with applications in algebraic
  complexity and quantum information.
\newblock In {\em S{TOC}'25---{P}roceedings of the 57th {A}nnual {ACM}
  {S}ymposium on {T}heory of {C}omputing}, pages 756--765. ACM, 2025.
\newblock arXiv:2510.08336.

\bibitem{WDGC2013_Science}
M.~Walter, B.~Doran, D.~Gross, and M.~Christandl.
\newblock Entanglement polytopes: Multiparticle entanglement from
  single-particle information.
\newblock {\em Science}, 340(6137):1205--1208, 2013.

\bibitem{WigdersonZuiddam2023}
A.~Wigderson and J.~Zuiddam.
\newblock Asymtotic spectra: {T}heory, {A}pplications, and {E}xtensions.
\newblock {\em Bulletin of the American Mathematical Society}.
\newblock to appear.

\end{thebibliography}

\appendix
\section{Proof of Lemma~\ref{lem:ui-convex=>g-convex}}
Suppose that unitarily invariant convex function $F:{\cal H}_n \to \overline{\RR}$ is written as $F = F_f$ 
for a symmetric convex function $f:\RR^n \to \overline{\RR}$.
We first observe ($*$1) if $\xi \in {\cal B}_n$ is 
represented by $\lambda_1 \geq \lambda_2 \geq \cdots \geq \lambda_n$, 
${\cal Y}_1 < {\cal Y}_2 <\cdots < {\cal Y}_n$, 
then $F(\xi) = f(\lambda)$.
Indeed, by the Gram-Schmidt orthonormalization, we choose an orthonormal basis $u_1,u_2,\ldots,u_n$ such that $\langle u_1,u_2,\ldots,u_j\rangle = {\cal Y}_j$ for $j \in [n]$. Then $\xi$ corresponds Hermitian matrix $k (\diag \lambda) k^{\dagger}$ 
for $k= (u_1\ u_2\ \cdots \ u_n) \in U_n$. Thus $F(\xi) = f(\lambda)$.

Since any two points of ${\cal B}_n$ belong to a common apartment, 
it suffices to show that $F$ is convex on each apartment.
For any unitary $k \in U_n$, 
the set ${\cal E}_k:=  \{ k (\diag x) k^{\dagger} \mid x \in \RR^n \}$ 
is one of apartments, in which $\RR^n \ni x \mapsto k (\diag x)s k^{\dagger}$ is an isometry.
Therefore, $F$ is convex on such an apartment ${\cal E}_k$.
Now consider a general apartment ${\cal A}$ for 
(non orthonormal) basis $v_1,v_2,\ldots,v_n$.
Choose an orthonormal basis $u_1,u_2,\ldots,u_n$ such that $\langle u_1,u_2,\ldots,u_j\rangle = \langle v_1,v_2,\ldots,v_j\rangle$ for $j \in [n]$. Let $k:= (u_1\ u_2\ \cdots \ u_n) \in U_n$.
Let $\rho:{\cal A} \to {\cal E}_{k}$ be  
defined by
\begin{equation*}
\begin{array}{c}
y_{i_1} \geq y_{i_2}\geq \cdots \geq y_{i_n}, \\
\langle v_{i_1}\rangle < \langle v_{i_1},v_{i_2} \rangle < \cdots <  \langle v_{i_1},v_{i_2},\ldots,v_{i_n} \rangle 
\end{array} \mapsto 
\begin{array}{c}
y_{i_1} \geq y_{i_2}\geq \cdots \geq y_{i_n}, \\
\langle u_{i_1}\rangle < \langle u_{i_1},u_{i_2} \rangle < \cdots <  \langle u_{i_1},u_{i_2},\ldots,u_{i_n} \rangle 
\end{array}.
\end{equation*}
It is clear that $\rho:{\cal A} \to {\cal E}_{k}$ is an isometry.
By ($*$1), it holds ($*$2) $F(\xi) = F \circ \rho(\xi) = f(y)$.
Therefore, for $\xi,\xi' \in {\cal A}$ and $t \in [0,1]$, 
denoting by $(1-t) \xi + t \xi'$ the point 
$\gamma(t)$ on the unique geodesic 
$\gamma:[0,1] \to {\cal B}_n$ with $\gamma(0) = \xi$ and $\gamma(1) = \xi'$, 
we have 
$F((1-t) \xi + t\xi') = F \circ \rho ((1-t)\xi + t \xi') = F((1-t) \rho(\xi) + t\rho (\xi')) 
\leq (1-t) F \circ \rho(\xi) + t F \circ \rho(\xi') = (1-t) F(\xi) + t F(\xi')$, 
where the first and last equalities follow from ($*$2), the second from the fact 
that $\rho$ is an isometry, and the third follows from the convexity of $F$ on ${\cal E}_k$.

\section{Proof of Proposition~\ref{prop:Q2=>Q3}}
For $x,y \in {\cal M}$, define a bijection $\sigma_{x \to y}:T_x \to T_y$
by $u \mapsto v$ where 
$t \mapsto \exp_x tu$ and $t \mapsto \exp_y tv$ are asymptotic.
Then $\sigma_{x \to y}$ and $\tau_{x\to y}$ are related as follows:  
\begin{Lem}\label{lem:sigma_tau}
For $x,y \in {\cal M}$ and $u \in T_x$, there is $k \in Hol_y({\cal M})$ such that
\begin{equation}\label{eqn:sigma_tau}
\tau_{x\to y} u = k \cdot \sigma_{x \to y} u.
\end{equation}
\end{Lem}
\begin{proof}
Via the de Rham decomposition, 
we may assume that ${\cal M}$ is irreducible.
The Berge holomony theorem implies that $Hol_x({\cal M})$ acts transitively on the sphere in $T_x$ 
or ${\cal M}$ is a symmetric space.
In the former case, since $\|\tau_{x\to y} u\| = \|\sigma_{x\to y} u\|$, 
there is $k \in Hol_y({\cal M})$ such that $\tau_{x\to y} u = k \cdot \sigma_{x \to y} u$.
So suppose that ${\cal M}$ is an irreducible symmetric space (of noncompact type).
It is known (see e.g., \cite[Chapter II.10]{BridsonHaefliger1999}) that ${\cal M} = G/K$ is realized in a totally geodesic subspace 
$P_{n}(\RR) \cap G$ of $P_{n}(\RR) := P_n \cap GL_n(\RR)$ for a connected semisimple Lie group $G \subseteq GL_n(\RR)$ with $K = G \cap O(n)$.
For an isometry $\varphi:M \to M$, we observe that $v = \sigma_{x \to y}v \Leftrightarrow 
d \varphi_y v = \sigma_{\varphi(x) \to \varphi(y)} d \varphi_x u$.
Also $\tau_{\varphi(x) \to \varphi(y)} = d\varphi_y \circ \tau_{x \to y} \circ d\varphi^{-1}_x$ 
and $d\varphi_y \circ k \circ d\varphi^{-1}_y \in Hol_{\varphi(y)}({\cal M})$ for $k \in Hol_{y}({\cal M})$.
Since the isometry $g,x \mapsto gxg^{\dagger}$ acts transitively on ${\cal M}$, it suffices to show (\ref{eqn:sigma_tau}) for a fixed $y  = I$.
Let $x \in {\cal M}$ and $H \in T_x$.
Consider generalized Iwasawa decomposition~\cite[2.17.5]{Eberlein} $x^{-1/2} = k b$, 
where $k \in K = Hol_I({\cal M})$ 
and $b \in G$ induces an isometry that fixes the asymptotic class of $H$, 
i.e.,  $\sigma_{x \to I}H = bHb^{\dagger}$.
Then $\tau_{x \to I} H = x^{-1/2}Hx^{-1/2} = k b H b^{\dagger} k^{\dagger} = k \cdot \sigma_{x \to I} H$.
\end{proof}

\begin{proof}[Proof of Proposition~\ref{prop:Q2=>Q3}]
Suppose that $Q$ satisfies Q2.
Let $x,y \in {\cal M}$ and $u \in T_x$. 
Choose $k \in Hol_y({\cal M})$ satisfying (\ref{eqn:sigma_tau}).
%We show $Q^*_x(-u)=Q^*_y(-\sigma_{x \to y}u) = $.
Then 
$Q^*_x(-u) = \sup_{p \in T_x^*}p(-u) - Q_x(p) = \sup_{p \in T_x^*}\tau_{x\to y}p(-\tau_{x \to y}u) - Q_y(\tau_{x \to y}p) = \sup_{q \in T_y^*} q(-k\sigma_{x \to y}u) - Q_y(q) = \sup_{q \in T_y^*} k^{-1}q(-\sigma_{x \to y}u) - Q_y(q) = 
\sup_{q' \in T_y^*} q'(-\sigma_{x \to y}u) - Q_y(q') = Q^*_y(-\sigma_{x\to y}u)$, 
where we used Q2 for the second and the last inequalities.
\end{proof}

\end{document}